\documentclass[sn-mathphys,Numbered]{sn-jnl} 
\usepackage{relsize}
\usepackage{lscape}
\usepackage{pdflscape}
\usepackage{mathtools}
\usepackage{cite}
\usepackage{anyfontsize}
\usepackage{graphicx} 
\usepackage{multirow} 
\usepackage{amsmath,amssymb,amsfonts} 
\usepackage{amsthm} 
\usepackage{mathrsfs} 
\usepackage[title]{appendix} 
\usepackage{xcolor} 
\usepackage{float}
\usepackage{textcomp} 
\usepackage{manyfoot} 
\usepackage{booktabs} 
\usepackage{algorithm} 
\usepackage{algorithmicx} 
\usepackage{algpseudocode} 
\usepackage{listings}
\usepackage{array}
\usepackage{caption}
\usepackage{subcaption}

\usepackage[utf8]{inputenc}
\usepackage{booktabs}
\usepackage{amsfonts}
\usepackage{siunitx}
\usepackage{longtable}
\usepackage{eucal}
\usepackage{multirow}
\usepackage{enumerate}
\usepackage{lineno}
\usepackage{framed}
\usepackage{placeins}
\usepackage{cleveref} 
\newtheorem{theorem}{Theorem}[section]
 
\newtheorem{lemma}{Lemma}[section]

\newtheorem{remark}{Remark}[section] 
\newtheorem{assumption}{Assumption}[section] 
\newtheorem{definition}{Definition}[section]
\newtheorem{property}{Property}

\DeclareMathOperator{\argmin}{\textnormal{argmin}}
\begin{document}

\title[PRP, HS and LS Conjugate Gradient Methods for Interval-Valued MOPs]{PRP, HS and LS  Conjugate Gradient Methods for Interval-Valued Multiobjective Optimization Problems}

\author[1]{\fnm{Tapas} \sur{Mondal}}\email{tapas.ra.mat24@itbhu.ac.in}
\author*[2]{\fnm{Debdulal} \sur{Ghosh}}\email{debdulag@srmist.edu.in, debdulal.email@gmail.com}
\author[3,4]{\fnm{Zai-Yun} \sur{Peng}}\email{pengzaiyun@126.com}
\author[5]{\fnm{Yong} \sur{Zhao}}\email{zhaoyongty@126.com}

\affil[1]{\orgdiv{Department of Mathematical Sciences}, \orgname{Indian Institute of Technology (BHU)}, \orgaddress{\city{Varanasi}, \postcode{221005}, \state{Uttar Pradesh}, \country{India}}}

\affil[2]{\orgdiv{Department of Mathematics}, \orgname{SRM Institute of Science and Technology}, \orgaddress{\city{Kattankulathur}, \postcode{603203}, \state{Tamil Nadu}, \country{India}}}

\affil[3]{\orgdiv{School of Mathematics}, \orgname{Yunnan Normal University}, \orgaddress{\city{Kunming}, \state{650092}, \country{China}}}

\affil[4]{\orgdiv{Yunnan Key Laboratory of Modern Analytical Mathematics and Applications}, \orgname{Yunnan Normal University}, \orgaddress{\city{Kunming}, \state{650092}, \country{China}}}

\affil[5]{\orgdiv{College of Mathematics and Statistics}, \orgname{Chongqing Jiaotong University}, \orgaddress{\city{Chongqing}, \postcode{400074}, \country{China}}}

\abstract{In this article, we develop an efficient algorithm based on three special variants of the nonlinear conjugate gradient method, namely, the Polak--Ribiere--Polyak, Hestenes--Stiefel, and Liu--Story schemes for computing Pareto critical points in unconstrained interval-valued multiobjective optimization problems. The proposed algorithm incorporates a Wolfe line search strategy to determine a suitable range of step size that satisfies the standard Wolfe conditions. For each of the proposed variants of the nonlinear conjugate gradient method, we establish rigorous global convergence results under appropriate assumptions. To demonstrate the effectiveness of the proposed methods, we conduct numerical experiments on a set of benchmark test problems and present a comprehensive performance profile analysis.}



\maketitle
\section{Introduction}\label{Introduction}
In real-world optimization tasks, it is common to deal with several objective functions at the same time, and these objectives often conflict with one another. Problems of this nature are called {\it multiobjective optimization problems} (MOPs). In contrast to single objective optimization, where a single best solution is typically obtained, MOPs lead to a collection of compromise solutions, known as Pareto optimal or nondominated solutions. Such problems arise naturally in various application areas, including engineering design, economics, and management science.

Standard formulations of MOPs generally assume that the objective functions are known precisely. However, in numerous real-life situations, this assumption may not hold due to incomplete, imprecise, or uncertain information. As a result, the objective functions may exhibit variability or uncertainty, motivating the study of MOPs under imprecise or uncertain environments. This has prompted to a growing research interest in developing efficient solution methodologies for such problems, particularly when the objective functions are represented using interval-valued models.

Among the various numerical techniques, iterative methods that utilize gradient information have proven to be highly effective for solving {\it interval-valued multiobjective optimization problems} (IVMOPs). In particular, {\it nonlinear conjugate gradient methods} (NCGMs) are widely recognized for their computational efficiency and strong global convergence behavior. In this work, we focus on the development of three specialized variants of the NCGM, namely the Polak–Ribiere–Polyak (PRP), Hestenes–Stiefel (HS), and Liu–Storey (LS) schemes. These variants are particularly appealing due to their superior computational performance compared to several classical approaches, including the Fletcher–Reeves (FR), Conjugate Descent (CD), Dai–Yuan (DY), and modified Dai–Yuan (mDY) methods.

Before proceeding further, we provide a brief overview of the existing literature on solution methodologies for MOPs.

\subsection{Literature Survey}
Over the past few decades, a variety of parameter-based techniques have been proposed for solving MOPs, including the weighted sum, $\epsilon$-constraint, and weighted Tchebycheff metric scalarization methods (see, e.g., \cite{ehrgott2005multicriteria, miettinen1999nonlinear, ghosh2014directed}). These approaches primarily rely on transforming a MOP into an equivalent single objective formulation through appropriate parameterization. Although the weighted sum method is straightforward to implement, it is well known that it may fail to capture certain Pareto optimal solutions when the associated objective functions are nonconvex. The $\epsilon$-constraint method, on the other hand, is capable of identifying Pareto optimal solutions even in nonconvex MOPs, however, its effectiveness depends critically on the careful selection of the $\epsilon$-vector within suitable bounds determined by the individual objective functions. Similarly, the weighted Tchebycheff method ensures the generation of Pareto optimal solutions but requires prior knowledge of the extreme values of the objective functions. A common limitation of these parameter-based methods is the requirement to appropriately choose parameter values, which often imposes a significant burden on the decision-maker. To address these challenges, several iterative algorithms have been developed that utilize gradient and Hessian information of the objective functions. The pioneering work was done by Fliege and Svaiter \cite{fliege2000steepest} in 2000, where an iterative scheme based on the steepest descent (SD) algorithm was introduced to compute Pareto critical points in unconstrained MOPs. They also extended their ideas to constrained MOPs using a feasible direction approach. Notably, both methods avoided reliance on scalarization strategies or preference-based ordering structures. Subsequently, Drummond and Iusem \cite{drummond2004projected} proposed a projected gradient algorithm tailored for vector optimization problems, and established its convergence behavior under both convex and nonconvex scenarios.

Later, Assun\c{c}\~{a}o et al. \cite{assuncao2021conditional} investigated the applicability of the Frank--Wolfe technique in constrained MOPs, assuming the feasible region to be convex and compact. Their analysis included both convergence guarantees and complexity estimates for the iterative process. Beyond first-order methods, attention has also been given to higher-order approaches. In this direction, Fliege et al. \cite{fliege2009newton} introduced the Newton method in 2009, demonstrating locally superlinear convergence under the assumptions of twice continuous differentiability and local strong convexity of the objective functions. They further showed that quadratic convergence can be achieved when additional Lipschitz continuity conditions are satisfied.

Despite its faster convergence characteristics, the Newton method may encounter difficulties when the Hessian matrices are not invertible. To overcome this drawback, Povalej \cite{polvaj2014quasi} proposed a quasi-Newton method that replaces exact Hessian with a suitable approximation, along with corresponding convergence analysis. More recently, Lapucci and Mansueto \cite{lapucci2023limited} introduced a limited memory variant of quasi-Newton methods for MOPs, incorporating a Wolfe line search to determine appropriate step sizes. Their results include global convergence as well as $R$-linear convergence to Pareto optimal solutions for strongly convex and twice continuously differentiable objective functions. Furthermore, Mohammadi and Cust\'{o}dio \cite{mohammadi2024trust} in 2024 presented a trust region method aimed at approximating the collection of Pareto critical points in MOPs.

In more recent developments, NCGMs have emerged as one of the most efficient classes of algorithms for MOPs due to their favorable computational cost and strong global convergence properties. A foundational contribution in this direction is due to Lucambio P{\'e}rez and Prudente \cite{perez2018nonlinear}, who systematically extended NCGMs to vector optimization. In their work, generalized versions of the standard Wolfe and strong Wolfe line search conditions were introduced for MOPs, and a corresponding Zoutendijk condition was established. Utilizing this framework, the authors proved global convergence of their algorithm without imposing restrictive assumptions on the algorithmic parameters. Furthermore, they established convergence results for several classical NCGMs, including FR, CD, DY, mDY, PRP, and HS methods.

Further investigations have focused extensively on analyzing and improving the computational performance of different conjugate gradient parameters for the MOPs. For instance, in \cite{Elboulqe2024explicit}, the authors proposed an algorithm that directly handles biobjective optimization problems without resorting to scalarization. Their approach constructs explicit descent directions satisfying a sufficient descent condition independent of line search and establishes global convergence to Pareto stationary points under an Armijo-type condition. In another study, Chen et al. \cite{chen2025PRP} developed a PRP variant of the NCGM for nonconvex vector optimization problems by introducing a nonnegative conjugate gradient parameter, avoiding the conventional truncation strategy. Under suitable descent conditions, they proved global convergence using both standard Wolfe and Armijo-like line search techniques. Moreover, in \cite{Hu2025modified}, a modified PRP variant of the NCGM was proposed, demonstrating improved computational efficiency compared to the classical PRP scheme.

More recently, attention has also been directed toward the LS variant of NCGMs for vector optimization problems. In 2022, Gon\c{c}alves et al. \cite{gonclaves2022study} introduced the LS variant and established global convergence results under sufficient descent conditions and the standard Wolfe line search. They further proposed a modified LS parameter and proved convergence under the strong Wolfe condition, as well as under an Armijo-like line search strategy. However, their methods do not guarantee a descent direction at every iteration in the vector sense, even when exact line search is employed. To address this limitation, Peng et al. \cite{Peng2025novel} proposed a modified LS variant of the NCGM in 2025, which ensures improved descent properties. The authors demonstrated the computational efficiency of their approach and established global convergence under relatively weaker assumptions.

\subsection{Motivation and Work Done}
A review of the existing literature reveals that a wide range of parameter-based, ordering-based, and iterative algorithms have been developed for solving traditional MOPs. Despite the availability of these methods for computing Pareto optimal solutions, their applicability is often limited by the assumption that the objective functions are precisely known. In many practical situations, however, such an assumption may not hold, as the objective functions may exhibit uncertainty and are more appropriately represented by closed and bounded intervals.


Subsequently, Mondal and Ghosh \cite{mondal2025steepest} proposed an SD algorithm for computing Pareto critical points of IVMOPs. In their framework, descent directions at non-Pareto critical points are obtained by solving a strongly convex quadratic subproblem, and the convergence properties of the proposed method were rigorously established. Building on this line of research, Mondal et al. \cite{mondal2026nonlinear} developed a NCGM for IVMOPs. They derived a Zoutendijk-type condition tailored to IVMOPs and established global convergence for a general class of algorithmic parameters. Furthermore, global convergence results were proved for specific parameter choices, including FR, CD, DY, and mDY methods.

However, to the best of our knowledge, the performance and convergence behavior of other important NCGMs, namely PRP, HS, and LS have not yet been investigated in the context of IVMOPs. Motivated by this gap, the present work aims to develop and analyze these three variants within the IVMOP framework. The main contributions of this study are summarized as follows:

\begin{enumerate}
	\item[(i)] We propose PRP, HS, and LS variants of the NCGM for computing Pareto critical points of IVMOPs.
	\item[(ii)] We establish global convergence results for each of the proposed variants under appropriate assumptions.
	\item[(iii)] We demonstrate the effectiveness of the proposed methods through numerical experiments and provide a comparative performance analysis using Dolan–Mor{\'e} performance profiles \cite{dolan2002benchmarking}.
\end{enumerate}

\subsection{Delineation}
The structure of this paper is organized as follows. In Section \ref{Preliminaries}, we present the necessary background on interval analysis along with key concepts related to IVMOPs. The NCGM tailored to the IVMOP is developed in Section \ref{Nonlinear Conjugate Gradient Method}. Section \ref{Convergence Analysis} is devoted to establish the convergence properties of the proposed method. Numerical performance is examined using benchmark test problems in Section \ref{Numerical Experiments}. Finally, Section \ref{Conclusion and Future Directions} concludes the paper with a summary of the results and directions for future research.

\section{Preliminaries}\label{Preliminaries}
This section introduces the basic principles of interval analysis and summarizes essential definitions along with preliminary results concerning IVMOPs. Throughout this work, the symbols ${\mathbb{R}}$ and ${\mathbb{N}}$ represent the collections of real and natural numbers, respectively.

\subsection{Interval Analysis}\label{interval analysis}
Let $\mathcal{I}\left(\mathbb{R}\right)$ denote the set of all compact intervals on ${\mathbb{R}}$.  
Let $\lambda \in \mathbb{R}$ and consider two intervals $P:=[p_1,p_2]$ and $Q:=[q_1,q_2]$ belonging to $\mathcal{I}\left(\mathbb{R}\right)$.  
Within the classical framework of interval arithmetic (see {\normalfont\cite{moore1966interval}}), we define the operations of addition, subtraction, and scalar multiplication for intervals, denoted by $P \oplus Q$, $P \ominus Q$, and $\lambda \odot P$, respectively. These operations are given by:
\begin{enumerate}
	\item[(i)] $P \oplus Q := \left[\,p_1+q_1, p_2+q_2\,\right]$;
	\item[(ii)] $P \ominus Q := \left[\,p_1-q_2, p_2-q_1\,\right]$;
	\item[(iii)] $\lambda \odot P :=
	\begin{cases}
		\left[\,\lambda p_1, \lambda p_2\,\right] & \text{if } \lambda \ge 0,\\[4pt]
		\left[\,\lambda p_2, \lambda p_1\,\right] & \text{if } \lambda < 0.
	\end{cases}
	$
\end{enumerate}

\medskip
\begin{definition}{\normalfont\cite{stefanini2008generalization}} \label{gH difference definition}
	\normalfont
	Let $P, Q \in \mathcal{I}(\mathbb{R})$. An interval $R \in \mathcal{I}(\mathbb{R})$ is said to represent the $gH$-difference of $P$ and $Q$ if it fulfills either of the equivalent conditions $P = Q \oplus R$ or $Q = P \ominus R$. In this situation, we denote
	\[
	R: = P \ominus_{gH} Q.
	\]
	For the particular case where $P:=[p_1,p_2]$ and $Q:=[q_1,q_2]$, the expression of the $gH$-difference is given by
	\[
	P \ominus_{gH} Q := \left[\min\left\{p_1 - q_1, p_2 - q_2\right\},\; \max\left\{p_1 - q_1, p_2 - q_2\right\}\right].
	\]
\end{definition}

\medskip
For a minimization interval optimization problem, to compare between two elements in the interval space $\mathcal{ I}\left({\mathbb{R}}\right)$, where smaller values get preference, we define the following ordering relation.

\medskip

\begin{definition} {\normalfont\cite{chauhan2021generalized}} \label{Dominance definition}
	\normalfont
	Let $P := [p_1,p_2]$ and $Q := [q_1,q_2]$ be members of $\mathcal{I}(\mathbb{R})$.
	\begin{enumerate}
		\item[(i)] We say that $Q$ dominates $P$ if $p_1 \ge q_1$ and $p_2 \ge q_2$. This is denoted by $P \succeq Q$.
		
		\item[(ii)] The dominance is strict if either $p_1 > q_1$ with $p_2 \ge q_2$, or $p_1 \ge q_1$ with $p_2 > q_2$. In this situation, we write $P \succ Q$.
		
		\item[(iii)] If the conditions in (i) are not fulfilled, then $P \nsucceq Q$.
		
		\item[(iv)] If the strict dominance conditions in (ii) fail, then $P \nsucc Q$.
		
		\item[(v)] The intervals $P$ and $Q$ are said to be comparable if either $P \succeq Q$ or $Q \succeq P$ holds.
		
		\item[(vi)] If neither $P$ dominates $Q$ nor $Q$ dominates $P$, then $P$ and $Q$ are called non-comparable.
	\end{enumerate}
\end{definition}

\medskip

The notation $P \succeq Q$ can also be written equivalently as $Q \preceq P$. In a similar way, $P \succ Q$, $P \nsucceq Q$, and $P \nsucc Q$ may be expressed as $Q \prec P$, $Q \npreceq P$, and $Q \nprec P$, respectively.

\medskip
\begin{definition}  {\normalfont\cite{moore1966interval}} \label{Norm on IR definition}
	\normalfont
	Assume that $\mathbb{R}_+$ denotes the set of all nonnegative real numbers. The function $\|\cdot\|_{\mathcal{I}(\mathbb{R})} : \mathcal{I}(\mathbb{R}) \to \mathbb{R}_+$ is called the norm on the interval space space $\mathcal{I}(\mathbb{R})$, which is defined for any element $P := [p_1, p_2] \in \mathcal{I}(\mathbb{R})$ by
	\[
	\left\|P\right\|_{\mathcal{I}(\mathbb{R})} := \max \left\{ \left|p_1\right|, \left|p_2\right| \right\}.
	\]
\end{definition}

\medskip
\noindent
Throughout this work, the notation $\|\cdot\|$ is used for the $l_2$-norm in $\mathbb{R}^n$, which is given by
\[\left\|x\right\|_2:=\sqrt{\sum_{j=1}^{n}x_j^2}\quad \text{ for all } x:=\left(x_1,x_2,\ldots,x_n\right)^\top\text{ in }{\mathbb{R}}^n.\]  
Let $\Phi : \mathbb{R}^n \to \mathcal{I}(\mathbb{R})$ be an IVF expressed as
\[
\Phi: =\left[\phi^{-}, \phi^{+}\right],
\]
where $\phi^{-}, \phi^{+} : \mathbb{R}^n \to \mathbb{R}$ are two functions such that
\[
\phi^{-}(x) \le \phi^{+}(x)  \text{ for all } x \in \mathbb{R}^n.
\]
The real-valued functions $\phi^{-}$ and $\phi^{+}$ are called the lower endpoint function and the upper endpoint function, respectively, corresponding to the IVF $\Phi$.

\medskip
\begin{definition}{\normalfont\cite{ghosh2017newton}} \label{gH continuity of IVF definition}
	\normalfont
	An IVF $\Phi : \mathbb{R}^n \to \mathcal{I}(\mathbb{R})$ is said to be $gH$-continuous at a point $x^\dagger \in \mathbb{R}^n$ if
	\[
	\underset{\left\|h\right\|\to 0}{\lim} \left( \Phi(x^\dagger + h) \ominus_{gH} \Phi(x^\dagger) \right) = \left[0,0\right].
	\]
\end{definition}

\medskip
\begin{definition} {\normalfont\cite{ghosh2022generalized}}
\label{gH Lipschitz continuity of IVF definition}
	\normalfont
	An IVF $\Phi : \mathbb{R}^n \to \mathcal{I}(\mathbb{R})$ is said to be $gH$-Lipschitz continuous with a Lipschitz constant $L > 0$ if
	\[
	\| \Phi(x) \ominus_{gH} \Phi(y) \|_{\mathcal{I}(\mathbb{R})} \le L \|x - y\| \quad \text{for every } x, y \in \mathbb{R}^n.
	\]
\end{definition}

\medskip
\begin{definition} {\normalfont\cite{debnath2022generalized}} \label{gH derivative of IVF definition}
	\normalfont
	Let $\Omega \subseteq \mathbb{R}$ be an open set and consider an IVF $\Phi : \Omega \to \mathcal{I}(\mathbb{R})$.  
	The  function $\Phi$ is said to be $gH$-differentiable at a point $x^{ \dagger} \in \Omega$ if the following limit is well-defined:
	\[
	\Phi'(x^{ \dagger}) :=
	\lim_{h \to 0} \tfrac{1}{h} \odot \left( \Phi(x^{ \dagger} + h) \ominus_{gH} \Phi(x^{ \dagger}) \right).
	\]
	Whenever this limit exists, it is referred to as the $gH$-derivative of $\Phi$ at $x^{ \dagger}$.
\end{definition}

\medskip
\begin{definition} {\normalfont\cite{debnath2022generalized}} \label{gH partial derivative of IVF definition}
	\normalfont
	Assume that $\Omega \subseteq \mathbb{R}^n$ is an open set and consider an IVF $\Phi: \Omega \to \mathcal{I}(\mathbb{R})$.  
	Fix a point $\bar{x} = (\bar{x}_1, \ldots, \bar{x}_n)^\top \in \Omega$. For each index $j \in \{1,2,\ldots,n\}$, define the  function of a single variable
	\[
	\Phi_j(t) := \Phi(\bar{x}_1, \ldots, \bar{x}_{j-1}, t, \bar{x}_{j+1}, \ldots, \bar{x}_n).
	\]
	If $\Phi_j$ is $gH$-differentiable at $t = \bar{x}_j$, then the $gH$-partial derivative of $\Phi$ corresponding to the $j$-th coordinate at $\bar{x}$ is defined by
	\[
	\partial^{gH}_j \Phi(\bar{x}) := \Phi_j'(\bar{x}_j), \quad j = 1,2, \ldots, n.
	\]
\end{definition}

\medskip
\begin{definition}{\normalfont\cite{debnath2022generalized}} \label{gH gradient of IVF definition}
	\normalfont
	Assume that $\Omega \subseteq \mathbb{R}^n$ is an open set and consider that $\Phi : \Omega \to \mathcal{I}(\mathbb{R})$ is an IVF.  
	At a point $\bar{x} \in \Omega$, the $gH$-gradient of $\Phi$, denoted by $\nabla_{gH} \Phi(\bar{x})$, is defined as the vector composed of all $gH$-partial derivatives of $\Phi$ at $\bar{x}$. That is,
	\[
	\nabla_{gH} \Phi(\bar{x}): = \left( \partial^{gH}_1 \Phi(\bar{x}), \partial^{gH}_2 \Phi(\bar{x}), \ldots, \partial^{gH}_n \Phi(\bar{x}) \right)^\top.
	\]
\end{definition}

\medskip
\begin{definition}{\normalfont\cite{ghosh2022generalized}} \label{Linear IVF definition}
	\normalfont
	Assume that $\Omega \subseteq \mathbb{R}^n$ is a vector subspace, and consider that $\Phi : \Omega \to \mathcal{I}(\mathbb{R})$ is an IVF.  
	The  function $\Phi$ is said to be linear if, for every vector $x := (x_1, x_2, \ldots, x_n)^\top \in \Omega$, it admits the representation
	\[
	\Phi(x) := \bigoplus_{k=1}^{n} \left( \Phi(e_k) \odot x_k \right),
	\]
	where $\left\{e_1,e_2,\ldots,e_n\right\}$ is the standard ordered basis of $\mathbb{R}^n$, and $\bigoplus$ indicates addition of successive intervals.
\end{definition}
\medskip
\begin{definition}{\normalfont\cite{ghosh2022generalized}}
\label{$gH$-differentiabe IVF definition}
	\normalfont
	Assume that $\Omega \subseteq \mathbb{R}^n$ is an open set and consider an IVF $\Phi : \Omega \to \mathcal{I}(\mathbb{R})$.  
	The  function $\Phi$ is said to be $gH$-differentiable at $x^ \dagger \in \Omega$ if there exist a linear IVF $L_{x^ \dagger} : \mathbb{R}^n \to \mathcal{I}(\mathbb{R})$, a remainder term $E\big(\Phi(x^ \dagger;h)\big)$, and a scalar $\delta > 0$ such that for all $h \in \mathbb{R}^n$ with $\|h\| < \delta$, the relation
	\[
	\Phi(x^ \dagger + h) \ominus_{gH} \Phi(x^ \dagger)
	:= L_{x^ \dagger}(h) \oplus \|h\| \odot E\left(\Phi(x^ \dagger;h)\right)
	\]
	holds, where the remainder term satisfies $E\left(\Phi(x^ \dagger;h)\right) \to [0,0]  \text{ as } \|h\| \to 0.$

	\medskip
	
	If the above condition is fulfilled at every point of $\Omega$, then $\Phi$ is called $gH$-differentiable on $\Omega$.
\end{definition}
\medskip
\begin{lemma}{\normalfont\cite{ghosh2022generalized}}
\normalfont
\label{linear IVF lemma}
	Assume that $\Omega \subseteq \mathbb{R}^n$ is an open set and consider that $\Phi : \Omega \to \mathcal{I}(\mathbb{R})$ is an IVF.  
	Suppose that $\Phi$ is $gH$-differentiable at $x^ \dagger \in \Omega$. Then, there exists a constant $\delta > 0$ such that, for any direction $h \in \mathbb{R}^n$ and any scalar $\alpha$ satisfying $|\alpha| \|h\| < \delta$, one has
	\[
	\lim_{\alpha \to 0} \tfrac{1}{\alpha} \odot \left( \Phi(x^ \dagger + \alpha h) \ominus_{gH} \Phi(x^ \dagger) \right)
	:= L_{x^ \dagger}(h),
	\]
	where $L_{x^ \dagger}$ denotes a linear IVF.
	
	\medskip
	
	Moreover, if the $gH$-gradient of $\Phi$ exists at a point $x^ \dagger$, namely
	\[
	\nabla_{gH} \Phi(x^ \dagger) :=
	\left( \partial^{gH}_1 \Phi(x^ \dagger), \partial^{gH}_2 \Phi(x^ \dagger), \ldots, \partial^{gH}_n \Phi(x^ \dagger) \right)^\top,
	\]
	then the  function $L_{x^ \dagger}$ can be written as
	\[
	L_{x^ \dagger}(h) := \nabla_{gH} \Phi(x^ \dagger)^\top \odot h
	= \bigoplus_{j=1}^{n} \partial^{gH}_j \Phi(x^ \dagger) \odot h_j
	\quad \text{for all } h := (h_1, h_2, \ldots, h_n)^\top \in \mathbb{R}^n.
	\]
\end{lemma}
\medskip

\begin{definition}{\normalfont\cite{wu2007karush}} 
\label{Convex IVF definition}
	\normalfont
	Assume that $\Omega \subseteq \mathbb{R}^n$ is a convex set and consider $\Phi : \Omega \to \mathcal{I}(\mathbb{R})$ is an IVF.  
	The  function $\Phi$ is said to be convex if, for any $x, y \in \Omega$ and for any scalar $\lambda \in [0,1]$, it satisfies
	\[
	\Phi\big(\lambda x + (1 - \lambda)y\big) \preceq \lambda \odot \Phi(x) \oplus (1 - \lambda) \odot \Phi(y).
	\]
\end{definition}

\medskip
\subsection{Interval-Valued Multiobjective Optimization Problem}
Assume that $\Lambda:=\left\{1,2,\ldots,m\right\}$ is the index set, which is used throughout the article.
Consider a vector-valued  function $G : \mathbb{R}^n \to \mathcal{I}(\mathbb{R})^m$ defined componentwise as
\[
G := \left(G_1, G_2, \ldots, G_m\right)^\top.
\]
The aim of this work is to study the IVMOP given by
\begin{align}\label{minG(x)}
	\min_{x \in \mathbb{R}^n} G(x).
\end{align}
In the IVMOP \eqref{minG(x)}, it is assumed that each component  function $G_i : \mathbb{R}^n \to \mathcal{I}(\mathbb{R})$ is $gH$-continuously differentiable, which are written in terms of their corresponding lower and upper endpoint functions as
\[
G_i := \left[\underline{G}_i, \overline{G}_i\right] \quad \text{for all } i \in \Lambda.
\]
Before developing the proposed methodology, we first recall the definitions of weak Pareto optimal, Pareto optimal, and Pareto critical points corresponding to the IVMOP \eqref{minG(x)}.

\medskip
\begin{definition}
{\normalfont\cite{mondal2025steepest}}
\label{weakly Pareto optimal}
	\normalfont
	A point $x^{ \star} \in \mathbb{R}^n$ is said to be a weak Pareto optimal of the IVMOP \eqref{minG(x)} if 
	\[
	\nexists\, x \in \mathbb{R}^n \text{ satisfying } 
	G_i(x) \prec G_i(x^{ \star}) \quad \text{for every } i \in \Lambda.
	\]
\end{definition} 

\medskip
\begin{definition} {\normalfont\cite{mondal2025steepest}}
\label{Pareto optimal}
	\normalfont
	A point $x^{ \star} \in \mathbb{R}^n$ is termed a Pareto optimal of the IVMOP \eqref{minG(x)} if 
	\[
	\nexists\, x \in \mathbb{R}^n \text{ satisfying } 
	G_i(x) \preceq G_i(x^{ \star}) \quad \text{for every } i \in \Lambda.
	\]
\end{definition}
\medskip

\begin{definition} {\normalfont\cite{mondal2025steepest}}
\label{Pareto critical}
	\normalfont
	A point $x^{ \dagger} \in \mathbb{R}^n$ is said to be a Pareto critical of the IVMOP \eqref{minG(x)} if 
	\[
	\nexists\, v \in \mathbb{R}^n \text{ satisfying }
	\nabla_{gH} G_i(x^{ \dagger})^\top \odot v \prec [0,0]
	\quad \text{for every } i \in\Lambda.
	\]
\end{definition}
\medskip
\begin{lemma}{\normalfont\cite{mondal2025steepest}}
\label{interrelation lemma of Pareto optimal and critical}
\normalfont
	Suppose that $G \in C_{gH}^1\left(\mathbb{R}^n, \mathcal{I}(\mathbb{R})^m\right)$, that is, for each $i\in\Lambda$, the component $G_i$ is $gH$-continuously differentiable.
	\begin{enumerate}
		\item[(i)]  
		If a point $x^{ \star}$ is a weak Pareto optimal of the IVMOP \eqref{minG(x)}, then it necessarily satisfies the conditions of Pareto criticality for the same problem.
		
		\item[(ii)]  
		Furthermore, suppose that for each index $i\in\Lambda$, the corresponding IVF $G_i$ is convex.  
		Then, any $x^{ \star} \in \mathbb{R}^n$ that is Pareto critical for the IVMOP \eqref{minG(x)} is also a weak Pareto optimal.
	\end{enumerate}
\end{lemma}
\medskip
\begin{definition} {\normalfont\cite{mondal2025steepest}}
\label{Descent direction}
	\normalfont
	A vector $v \in \mathbb{R}^n$ is said to satisfy descent condition for the IVMOP \eqref{minG(x)} at a fixed point $x \in \mathbb{R}^n$ if there exists a scalar $\delta > 0$ such that
	\[
	G_i(x + t v) \prec G_i(x) \quad \text{for every } i \in \Lambda \text{ and for every }t \in (0, \delta).
	\]
	The vector $v \in \mathbb{R}^n$ satisfying the above decency condition is called the descent direction for each IVF $G_i$ at the given point $x \in \mathbb{R}^n$.
\end{definition}

\medskip
Note that if $v$ is a descent direction for each $gH$-differentiable IVF $G_i$, $i \in \Lambda$, then it follows from Lemma \ref{linear IVF lemma} that
\begin{align*}
	\nabla_{gH} G_i(x)^\top\odot v:=\underset{t\to0}{\lim} \tfrac{1}{t}\odot\left[G_i(x+tv)\ominus_{gH} G_i(x)\right]\prec[0,0] \text{ for every } i\in\Lambda.
\end{align*}
In order to identify a vector $v \in \mathbb{R}^n$ that produces descent for the IVMOP \eqref{minG(x)} at a given point $x \in \mathbb{R}^n$, one aims to find $v$ such that
\[\nabla_{gH} G_i(x)^\top\odot v\prec[0,0] \text{ for every }i\in\Lambda.\]
To describe these directions at a fixed point $x$, we define, for every $i \in \Lambda$, an associated IVF $g_x^i : \mathbb{R}^n \to \mathcal{I}(\mathbb{R})$ as follows.
\begin{align}\label{gxi}
	g_x^i(v):=\nabla_{gH} G_i(x)^\top\odot v.
\end{align}
For each $j$-th coordinate, the corresponding component of the $gH$-gradient of $G_i$ at $x$ is denoted by
\[
\left[\underline{\nabla_{gH}G_i}(x)_j,\overline{\nabla_{gH}G_i}(x)_j\right]
:=\left[\min\left\{\frac{\partial \underline{G}_i(x)}{\partial x_j},\frac{\partial \overline{G}_i(x)}{\partial x_j}\right\},
\max\left\{\frac{\partial \underline{G}_i(x)}{\partial x_j},\frac{\partial \overline{G}_i(x)}{\partial x_j}\right\}\right].
\]
Next, consider the IVF $g_x^i:{\mathbb{R}}^n\to \mathcal{ I}\left({\mathbb{R}}\right)$ described through its endpoint  functions $\underline{g}_x^i$ and $\overline{g}_x^i$. Let $\left|v\right|:=\left(\left|v_1\right|,\left|v_2\right|,\ldots,\left|v_n\right|\right)^\top$. As established in \cite{mondal2025steepest}, these endpoint functions admit the compact representation
\begin{equation}\label{gxi-lower-upper}
	\begin{rcases}
		\begin{aligned}
			\underline{g}_x^i(v) &:= \tfrac{1}{2}\left(\underline{\nabla_{gH} G_i}(x)+\overline{\nabla_{gH} G_i}(x)\right)^\top v 
			-\tfrac{1}{2}\left(\overline{\nabla_{gH} G_i}(x)-\underline{\nabla_{gH} G_i}(x)\right)^\top |v|,\\
			\overline{g}_x^i(v) &:= \tfrac{1}{2}\left(\underline{\nabla_{gH} G_i}(x)+\overline{\nabla_{gH} G_i}(x)\right)^\top v 
			+\tfrac{1}{2}\left(\overline{\nabla_{gH} G_i}(x)-\underline{\nabla_{gH} G_i}(x)\right)^\top |v|.
		\end{aligned}
	\end{rcases}
\end{equation}

\medskip

Define the function $\Psi:{\mathbb{R}}^m\to {\mathbb{R}}$ by
\[
\Psi(z):=\max_{i\in\Lambda} ~z_i\quad\text{for all }z:=\left(z_1,z_2,\ldots,z_m\right)^\top \in {\mathbb{R}}^m.
\]
For a fixed $x\in{\mathbb{R}}^n$, introduce $\Upsilon_x:{\mathbb{R}}^n\to {\mathbb{R}}^m$ as
\[
\Upsilon_x(v):=\left(\overline{g}_x^1(v),\overline{g}_x^2(v),\ldots,\overline{g}_x^m(v)\right)^\top.
\]
Consequently, the composite  function $\Psi\circ\Upsilon_x:{\mathbb{R}}^n\to{\mathbb{R}}$ is expressed as
\begin{align}\label{psi-phi def}
	\Psi\circ\Upsilon_x(v):=\max_{i\in\Lambda}~\overline{g}_x^i(v).
\end{align}

\medskip

We now consider the optimization problem
\begin{align}\label{unconstrained min}
	\min_{v\in{\mathbb{R}}^n}\left(\Psi\circ\Upsilon_x(v)+\tfrac{1}{2}\|v\|^2\right).
\end{align}
The objective function in \eqref{unconstrained min} is strongly convex (see, for example, Lemma 3.1 in \cite{mondal2025steepest}), which guarantees the existence of a unique optimal solution.

At a fixed point $x$, assume that $v(x)$ and $\xi(x)$ denote the unique optimal solution and the optimal objective value of the unconstrained optimization problem \eqref{unconstrained min}, respectively, that is,
\begin{align}\label{v(x) and xi(x)} 
	v(x):=\underset{v\in{\mathbb{R}}^n}{\argmin}\: \left( \Psi\circ\Upsilon_x(v)+\tfrac{1}{2}\|v\|^2\right) \text{ and } \xi(x):=\underset{v\in{\mathbb{R}}^n}{\min} \left( \Psi\circ\Upsilon_x(v)+\tfrac{1}{2}\|v\|^2\right).
\end{align}

As shown in \cite{mondal2025steepest}, the computation of $v(x)$ can be reformulated as the following constrained quadratic program:
\begin{equation}\label{equivalent constrained problem}
	\begin{rcases}
		\begin{aligned}
			\underset{u,v\in{\mathbb{R}}^n,\tau\in{\mathbb{R}}}{\min}&\left(\tau+\tfrac{1}{2}\|v\|^2\right)\\
			\text{subject to } &\left(\underline{\nabla_{gH} G_i}(x)+\overline{\nabla_{gH} G_i}(x)\right)^\top v+\left({\overline{\nabla_{gH} G_i}(x)-\underline{\nabla_{gH} G_i}(x)}\right)^\top u\leq2\tau, i\in\Lambda,\\
			&-u_j\leq v_j \leq u_j, j=1,2,\ldots n.
		\end{aligned}
	\end{rcases}
\end{equation}

\medskip

\noindent
The vector $v(x)$ obtained from \eqref{equivalent constrained problem} plays an important role in identifying Pareto criticality for the IVMOP \eqref{minG(x)}. In particular, the quantities $\|v(x)\|$ and $\xi(x)$ provide useful indicators for detecting such critical points.
\medskip
\begin{lemma}\label{descent direction finding lemma}
\normalfont
\cite{mondal2025steepest}
	Assume that $v(x) \text{ in } \mathbb{R}^n$ is the unique optimal solution of the unconstrained optimization problem \eqref{unconstrained min}, i.e.,
	\[v(x):=\underset{v\in{\mathbb{R}}^n}{\argmin}\: \left( \Psi\circ\Upsilon_x(v)+\tfrac{1}{2}\|v\|^2\right),\]
	and denote by $\xi(x)$ the corresponding optimal value:
	\[
	\xi(x) := \min_{v \in \mathbb{R}^n} \left( \Psi \circ \Upsilon_x(v) + \tfrac{1}{2}\|v\|^2 \right).
	\]
	Then, the following assertions hold:
	\begin{enumerate}
		\item[(i)] For any $x \in \mathbb{R}^n$, it holds that $\xi(x) \le 0$.
		
		\item[(ii)] If $x$ is a Pareto critical point of the IVMOP \eqref{minG(x)}, then $v(x)$ coincides with the zero vector in $\mathbb{R}^n$, and consequently $\xi(x) = 0$.
		
		\item[(iii)] If $x$ is not a Pareto critical point of the IVMOP \eqref{minG(x)}, then $\xi(x) < 0$, and in this case $v(x) \ne 0$ provides a descent direction for each IVF $G_i$ of the IVMOP \eqref{minG(x)} at $x$.
		
		\item[(iv)] The  function $v:{\mathbb{R}}^n\to {\mathbb{R}}^n$ is bounded on any compact subset of $\mathbb{R}^n$.
		
		\item[(v)] The  function $\xi:{\mathbb{R}}^n\to {\mathbb{R}}$ is continuous over $\mathbb{R}^n$.
	\end{enumerate}
\end{lemma}
\medskip
\begin{remark}
	\normalfont
	According to Lemma \ref{descent direction finding lemma}, if $\xi(x)=0$ or equivalently $\|v(x)\|=0$, then $x$ satisfies the condition of Pareto criticality. On the other hand, whenever $\xi(x)\neq 0$ or $\|v(x)\|\neq 0$, the vector $v(x)$ provides a descent direction for every IVF $G_i$ of the IVMOP \eqref{minG(x)} at the point $x$.
\end{remark}
\medskip
The next lemma plays a key role in analyzing the convergence behavior of the proposed algorithm.
\medskip
\begin{lemma}
\normalfont \cite{perez2018nonlinear}
	Assume that $p,q,\text{ and }\alpha$ are three scalars with $\alpha\neq0$. Then, the following results hold: 
	\begin{enumerate}\label{inequality lemma}
		\item[(i)] $2pq\leq2\alpha^2p^2+\tfrac{q^2}{2\alpha^2}$,
		\item[(ii)] $\left(p+q\right)^2\leq\left(1+2\alpha^2\right)p^2+\left(1+\tfrac{1}{2\alpha^2}\right)q^2,$
		\item[(iii)] $\left(p+q\right)^2\leq 2p^2+2q^2.$
	\end{enumerate}
\end{lemma}

\section{Nonlinear Conjugate Gradient Method}\label{Nonlinear Conjugate Gradient Method}
This section introduces the construction of three particular versions of the NCGMs, namely PRP, HS, and LS, with the aim of computing a Pareto critical point of the IVMOP \eqref{minG(x)}. For each of these variants, the step size is determined using the standard Wolfe conditions adapted to the IVMOP \eqref{minG(x)}.
\medskip
\begin{definition}
{\normalfont\cite{mondal2026nonlinear}}
\label{Standard and strong Wolfe conditions definition}
	\normalfont
	Assume that $d \in \mathbb{R}^n$ is a descent direction for every IVF $G_i$ of the IVMOP \eqref{minG(x)} at a point $x$. Fix constants $0 < \rho < \sigma < 1$. A scalar $t > 0$ is said to satisfy the standard Wolfe conditions if
	\begin{equation}\label{Standard Wolfe condition 1}
		G_i(x + td) \preceq G_i(x) \oplus [\rho t, \rho t] \odot \Psi \circ \Upsilon_x(d) \quad \text{for every } i \in \Lambda
	\end{equation}
	and
	\begin{equation}\label{Standard Wolfe condition 2}
		\Psi \circ \Upsilon_{x+td}(d) \ge \sigma \, \Psi \circ \Upsilon_x(d).
	\end{equation}
\end{definition}
\medskip
The next result guarantees the existence of an interval of step sizes for which the standard Wolfe conditions \eqref{Standard Wolfe condition 1}--\eqref{Standard Wolfe condition 2} are fulfilled.
\medskip
\begin{lemma}{\normalfont\cite{mondal2026nonlinear}}\label{existence of an interval of steplength}
\normalfont
	Suppose that $G_i$ is $gH$-continuously differentiable for every $i \in \Lambda$, and let $d \in \mathbb{R}^n$ be a descent direction for every IVF $G_i$ at a point $x$. Assume further that, for every $i \in \Lambda$, there exists an interval $A_i \in \mathcal{I}(\mathbb{R})$ satisfying the following dominance relation
	\[
	G_i(x + td) \succeq A_i \quad \text{for every } t > 0.
	\]
	Then, there exists a nonempty interval of step sizes $t$ for which the standard Wolfe conditions \eqref{Standard Wolfe condition 1}--\eqref{Standard Wolfe condition 2} are satisfied.
\end{lemma}

\medskip

We now present a step-by-step procedure of the NCGM aimed at computing Pareto critical points of the IVMOP \eqref{minG(x)} without restriction on the algorithmic parameter.
\medskip
\begin{algorithm}[H]
	\caption{NCGM for computing Pareto critical points of the IVMOP \eqref{minG(x)} \label{Algorithm}}
	\begin{enumerate}[\bf{Step} 1 ]
		
		\item (Input data)\\
		For every $i\in\Lambda$, specify the $gH$-continuously differentiable IVF $G_i$.
		
		\item (Initialization)\\
		Choose parameters $\rho$ and $\sigma$ such that $0<\rho<\sigma<1$, and select an initial point $x^0 \in \mathbb{R}^n$. Fix a tolerance $\epsilon > 0$ and set $k=0$.
		
		\item (Evaluation of $gH$-gradients at $x^k$)\\
		Compute
		\[
		\nabla_{gH} G_i(x^k) := \left[\underline{\nabla_{gH}G_i}(x^k), \overline{\nabla_{gH}G_i}(x^k)\right] \text{ for every } i \in \Lambda.
		\]
		
		\item (Determination of a descent direction at $x^k$)\\
		Obtain $v(x^k)$ and $\xi(x^k)$ by solving \eqref{unconstrained min}, namely,
		\[
	v\left(x^k\right):=\underset{v\in{\mathbb{R}}^n}{\argmin}\: \left( \Psi\circ\Upsilon_{x^k}(v)+\tfrac{1}{2}\|v\|^2\right)\text{ and }
	\xi\left(x^k\right):=\underset{v\in{\mathbb{R}}^n}{\min}\: \left( \Psi\circ\Upsilon_{x^k}(v)+\tfrac{1}{2}\|v\|^2\right).
	\]

		\item (Termination check)\\
		If $\xi(x^k) > -\epsilon$, stop the iteration and declare $x^k$ as a Pareto critical point. Otherwise, continue to {\bf Step 6}.
		
		\item (Construction of conjugate direction)\\
		Define
		\begin{align}\label{conjugate direction in algorithm}
			d^k :=
			\begin{cases}
				v(x^k) & \text{if } k = 0,\\
				v(x^k) + \beta_k d^{k-1} & \text{if } k \ge 1,
			\end{cases}
		\end{align}
		where $\beta_k$ is a chosen parameter.
		
		\item (Step size selection)\\
		Determine $t_k > 0$ such that
		\begin{align}\label{Standard Wolfe condition in algorithm}
			\begin{cases}
				G_i(x^k + t_k d^k) \preceq G_i(x^k) \oplus [\rho t_k, \rho t_k] \odot \Psi \circ \Upsilon_{x^k}(d^k) & \text{ for every } i \in \Lambda,\\
				\Psi \circ \Upsilon_{x^k + t_k d^k}(d^k) \ge \sigma \, \Psi \circ \Upsilon_{x^k}(d^k).
			\end{cases}
		\end{align}
		
		\item (Iteration update)\\
		Set $x^{k+1} := x^k + t_k d^k$, increment $k \leftarrow k+1$, and return to {\bf Step 3}.
		
	\end{enumerate}
\end{algorithm}
The computations of {\bf Step 4} and {\bf Step 7} are two main important parts of Algorithm \ref{Algorithm}. Since the objective function of the unconstrained optimization problem \eqref{unconstrained min} is strongly convex, it possesses a unique optimal solution. As a result, the vector $v(x^k)$ is well-defined, which guarantees the correctness of {\bf Step 4} in Algorithm \ref{Algorithm}.  

Furthermore, Lemma \ref{existence of an interval of steplength} ensures that there exists a nonempty range of step sizes $t_k$ satisfying the standard Wolfe conditions given in \eqref{Standard Wolfe condition in algorithm}. However, the applicability of this lemma requires that $d^k$ be a descent direction for every IVF $G_i$ at $x^k$. It follows that $d^k$ satisfies the following condition: \[\Psi\circ\Upsilon_{x^k}(d^k) < 0.\]  

For the study of convergence behavior of Algorithm \ref{Algorithm}, we impose the stronger requirement
\begin{align}\label{sufficient descent condition}
	\Psi\circ\Upsilon_{x^k}(d^k) \le c \, \Psi\circ\Upsilon_{x^k}(v(x^k)) \quad \text{for some } c > 0 \text{ and all } k \ge 0.
\end{align}
This requirement is referred to as the sufficient decency condition. It can be enforced through a suitable line search procedure, provided that $d^{k-1}$ is a descent direction for all IVF $G_i$ at $x^{k-1}$. Under this premise, Lemma \ref{existence of an interval of steplength} guarantees that performing a line search along $d^{k-1}$ using the standard Wolfe conditions yields the next iterate $x^k$. For the search direction, we have
\[
\Psi\circ\Upsilon_{x^k}(d^k) = \Psi\circ\Upsilon_{x^k}\big(v(x^k) + \beta_k d^{k-1}\big)
\le \Psi\circ\Upsilon_{x^k}(v(x^k)) + \beta_k \, \Psi\circ\Upsilon_{x^k}(d^{k-1}).
\]
For a bounded sequence $\{\beta_k\}$, the standard Wolfe line search can sufficiently decrease the quantity $|\Psi\circ\Upsilon_{x^k}(d^{k-1})|$, thereby ensuring that the condition \eqref{sufficient descent condition} is satisfied.
\section{Convergence Analysis}\label{Convergence Analysis}
In this section, we analyze the convergence of Algorithm \ref{Algorithm} related to the PRP, HS, and LS parameters given by
\begin{align*}
	&\beta_k^{PRP}:=\frac{-\Psi\circ\Upsilon_{x^k}\left(v\left(x^k\right)\right)+\Psi\circ\Upsilon_{x^{k-1}}\left(v\left(x^k\right)\right)}{-\Psi\circ\Upsilon_{x^{k-1}}\left(v\left(x^{k-1}\right)\right)},\\
&\beta_k^{HS}:=\frac{-\Psi\circ\Upsilon_{x^k}\left(v\left(x^k\right)\right)+\Psi\circ\Upsilon_{x^{k-1}}\left(v\left(x^k\right)\right)}{\Psi\circ\Upsilon_{x^{k}}\left(d^{k-1}\right)-\Psi\circ\Upsilon_{x^{k-1}}\left(d^{k-1}\right)},\\
\text{ and }&\beta_k^{LS}:=\frac{-\Psi\circ\Upsilon_{x^k}\left(v\left(x^k\right)\right)+\Psi\circ\Upsilon_{x^{k-1}}\left(v\left(x^k\right)\right)}{-\Psi\circ\Upsilon_{x^{k-1}}\left(d^{k-1}\right)}.
\end{align*}
To study the global convergence properties of Algorithm \ref{Algorithm} related to the PRP, HS, and LS parameters, we consider the following assumptions.
\medskip
\begin{framed}
	\begin{assumption}\label{assumption 1}
		\normalfont
		For every $i\in\Lambda$, $\nabla_{gH}G_i$ is $gH$-Lipschitz continuous with Lipschitz constant $L_i$ on the level set \[L_0:=\left\{x\in{\mathbb{R}}^n:G_i\left(x\right)\preceq G_i\left(x^0\right)\text{ for every }i\in\Lambda\right\}.\]
	\end{assumption}
	\medskip
	\begin{assumption}\label{assumption 2}
		\normalfont
		The level set $L_0:=\left\{x\in{\mathbb{R}}^n:G_i\left(x\right)\preceq G_i\left(x^0\right)\text{ for every }i\in\Lambda\right\}$ is bounded.
	\end{assumption}
\end{framed}
Next, to study the global convergence of PRP, HS, and LS variants of the NCGM, we define the following property.
\begin{framed}
	\renewcommand{\theproperty}{($*$)}
\begin{property}\label{property}
Consider the Algorithm \ref{Algorithm} and suppose that 
\begin{equation}\label{property eq 1}
	0<\gamma\leq \left\|v\left(x^k\right)\right\|\leq \bar{\gamma} \text{ for all } k\geq0.
\end{equation}
Under this assumption, we say that the method has Property \ref{property} if there exist constants $b>1$ and $\lambda>0$ such that, for all $k$, 
\[\beta_k\leq b\]
and 
\[\left\|s^{k-1}\right\|\leq\lambda\implies\left|\beta_k\right|\leq\tfrac{1}{2b},\text{ where }s^{k-1}:=x^k-x^{k-1}.\]
\end{property}
\renewcommand{\theproperty}{\arabic{property}}
\end{framed}
\medskip
First, we prove that if Algorithm \ref{Algorithm} satisfies Property \ref{property} with the sufficient descent condition \eqref{sufficient descent condition} and the standard Wolfe condition \eqref{Standard Wolfe condition in algorithm}, then it has a global convergence under Assumptions \ref{assumption 1} and \ref{assumption 2}. To prove this result, the following two lemmas are useful.
\medskip
\begin{lemma}\label{convergence proof lemma1}
\normalfont
	Suppose that Assumptions \ref{assumption 1} and \ref{assumption 2} are satisfied.  
	Consider Algorithm \ref{Algorithm} with $\beta_k \ge 0$, where $d^k$ is a sufficient descent direction for $G$ at $x^k$, and the step length $t_k$ fulfills the standard Wolfe conditions \eqref{Standard Wolfe condition in algorithm}.  
	Further, suppose that there exists a constant $\gamma > 0$ such that $\|v(x^k)\| \ge \gamma$ for all $k \ge 0$. Then, the following statements hold:
	\begin{itemize}
		\item[{\normalfont (i)}] $\underset{k\geq0}{\mathlarger\sum}\tfrac{\left\|v\left(x^k\right)\right\|^4}{\left\|d^k\right\|^2}<+\infty$,
		\item[{\normalfont (ii)}] $\underset{k\geq1}{\mathlarger\sum}\left\|u^k-u^{k-1}\right\|^2<+\infty$, where $u^k:=\tfrac{d^k}{\left\|d^k\right\|}$.
	\end{itemize}
\end{lemma}
\medskip
\begin{proof}
	Since $d^k$ is a sufficient descent direction of $G$ at $x^k$, we get $d^k\neq0$. Therefore, $\tfrac{\left\|v\left(x^k\right)\right\|^4}{\left\|d^k\right\|^2}$ and $u^k$ are well-defined. The proof of part (i) can be obtained from the proof of Theorem 4.1(i) given in \cite{mondal2026nonlinear}.
	
	Now we consider part (ii). Let us define $r^k:=\tfrac{v\left(x^k\right)}{\left\|d^k\right\|}$ and $\delta_k:=\tfrac{\left\|d^{k-1}\right\|}{\left\|d^{k}\right\|}$. Note that $u^k=r^k+\delta_k u^{k-1}$ and $\left\|u^k\right\|=\left\|u^{k-1}\right\|$. We have 
	\begin{align*}
		\left\|r^k\right\|^2&=\left\|u^k-\delta_k u^{k-1}\right\|^2\\
		&=\left\|u^k\right\|^2-2\delta_k\langle u^k,u^{k-1}\rangle+\delta_k^2\left\|u^{k-1}\right\|^2\\
		&=\left\|u^{k-1}\right\|^2-2\delta_k\langle u^k,u^{k-1}\rangle+\delta_k^2\left\|u^{k}\right\|^2\\
		&=\left\|\delta_k u^k- u^{k-1}\right\|^2.
	\end{align*}
	Therefore, we get $\left\|r^k\right\|=\left\|u^k-\delta_k u^{k-1}\right\|=\left\|\delta_k u^k- u^{k-1}\right\|.$
	Since $\delta_k\geq0$, we get 
	\begin{align*}
		\left\|u^k-u^{k-1}\right\|&\leq \left(1+\delta_k\right)\left\|u^k- u^{k-1}\right\|\\
		&=\left\|u^k- \delta_k u^{k-1}+\delta_k u^k-u^{k-1}\right\|\\
		&\leq\left\|u^k- \delta_k u^{k-1}\right\|+\left\|\delta_k u^k-u^{k-1}\right\|\\
	    &\leq\left\|r^k\right\|+\left\|r^k\right\|=2\left\|r^k\right\|.
	\end{align*}
	Using part (i), we get 
	\[\gamma^2 \underset{k\geq1}{\mathlarger\sum}\left\|u^k-u^{k-1}\right\|^2\leq 4\gamma^2 \underset{k\geq1}{\mathlarger\sum}\left\|r^k\right\|^2\leq 4 \underset{k\geq1}{\mathlarger\sum}\left\|r^k\right\|^2 \left\|v\left(x^k\right)\right\|^2=4 \underset{k\geq1}{\mathlarger\sum}\tfrac{\left\|v\left(x^k\right)\right\|^4}{\left\|d^k\right\|^2}<+\infty. \]
	Therefore, we get $\underset{k\geq1}{\mathlarger\sum}\left\|u^k-u^{k-1}\right\|^2< +\infty$, and this completes the proof.
\end{proof}
\medskip
For $\lambda>0$ and a positive integer $\Delta$, define 
\[\mathcal{K}_{k,\Delta}^\lambda:=\left\{i\in{\mathbb{N}}:k\leq i\leq k+\Delta-1, \left\|s^{k-1}\right\|>\lambda\right\}\]
and denote by $\left|\mathcal{K}_{k,\Delta}^\lambda\right|$ the number of elements of $\mathcal{K}_{k,\Delta}^\lambda$.
\medskip
\begin{lemma}\label{convergence proof lemma2}
	\normalfont
    Suppose that Assumptions \ref{assumption 1} and \ref{assumption 2} are satisfied.  
	Consider Algorithm \ref{Algorithm} in which $d^k$ is a sufficient descent direction for $G$ at $x^k$, and the step size $t_k$ satisfies the standard Wolfe conditions \eqref{Standard Wolfe condition in algorithm}.  
	Further, assume that the method has Property \ref{property}, and that there exists a constant $\gamma > 0$ such that $\|v(x^k)\| \ge \gamma$ for all $k \ge 0$.  
	
	Then, there exists a constant $\lambda > 0$ such that, for any $\Delta \in \mathbb{N}$ and any index $k_0$, one can find an index $k \ge k_0$ for which
	\[
	\left|\mathcal{K}_{k,\Delta}^\lambda\right| > \tfrac{\Delta}{2}.
	\]
\end{lemma}
\medskip
\begin{proof}
	If possible, suppose that for any $\lambda>0$, there exists $\Delta\in{\mathbb{N}}$ and $k_0$ such that for any $k\geq k_0$, we have \[\left|\mathcal{K}_{k,\Delta}^\lambda\right|\leq\tfrac{\Delta}{2}.\]	
	For any index $l\geq k_0+1$, using Lemma \ref{inequality lemma} (iii), we get 
	\[\left\|d^l\right\|^2\leq \left(\left\|v\left(x^l\right)\right\|+\left|\beta_l\right|\left\|d^{l-1}\right\|\right)^2\leq 2\left\|v\left(x^l\right)\right\|^2+2\beta_l^2\left\|d^{l-1}\right\|^2\leq 2\bar{\gamma}^2+2\beta_l^2\left\|d^{l-1}\right\|^2 \text{ for all } l\in{\mathbb{N}}.\]
	By the induction method, we obtain that 
	\[\left\|d^l\right\|^2\leq \bar{c}\left(1+2\beta_l^2+2\beta_l^2 2\beta_{l-1}^2+\cdots+2\beta_l^2 2\beta_{l-1}^2\cdots2\beta_{k_0}^2\right),\]
	where $\bar{c}$ depends on $\left\|d^{k_0-	1}\right\|$, but not on the index $l$.
	Proceeding with the similar arguments as in \cite{Gilbert1992global}, we get 
	\begin{align}\label{contradiction proof eq}
		\left\|d^l\right\|^2\leq \bar{c}\left(l-k_0+2\right) \text{ for all } l\geq k_0+1.
		\end{align}
	From part (i) of Lemma \ref{convergence proof lemma1}, we have 
	\[\gamma^4\underset{k\geq0}{\mathlarger\sum}\tfrac{1}{\left\|d^k\right\|^2}\leq \underset{k\geq0}{\mathlarger\sum}\tfrac{\left\|v\left(x^k\right)\right\|^4}{\left\|d^k\right\|^2}<+\infty,\]
	which contradicts \eqref{contradiction proof eq}. Hence, the proof is completed.
\end{proof}
\medskip
\begin{theorem}\label{convergence theorem}
	\normalfont
    Suppose that Assumptions \ref{assumption 1} and \ref{assumption 2} are valid.  
	Consider Algorithm \ref{Algorithm} with $\beta_k \ge 0$, where $d^k$ acts as a sufficient descent direction for $G$ at $x^k$, and the step size $t_k$ fulfills the standard Wolfe conditions \eqref{Standard Wolfe condition in algorithm}.  
	Further, suppose that Algorithm \ref{Algorithm} satisfies Property \ref{property}. Then,
	\[
	\liminf_{k \to \infty} \|v(x^k)\| = 0.
	\]
\end{theorem}
\medskip
\begin{proof}
	 	If possible, suppose that there exists a constant $\gamma>0$ such that 
	 \[\left\|v\left(x^k\right)\right\|\geq \gamma\text{ for all } k\geq 0.\]
	 Let us define $u^i:=\tfrac{d^i}{\left\|d^i\right\|}$. For any two indices $l$ and $k$ with $l\geq k$, we get 
	 \begin{align*}
	 	x^l-x^{k-1}&=\sum_{i=k}^{l}t_{i-1}d^{i-1}=\sum_{i=k}^{l}\left\|t_{i-1}d^{i-1}\right\|\tfrac{d^{i-1}}{\left\|d^{i-1}\right\|}=\sum_{i=k}^{l}\left\|s^{i-1}\right\|u^{i-1}\\
	 	&=\sum_{i=k}^{l}\left\|s^{i-1}\right\|u^{k-1}+\sum_{i=k}^{l}\left\|s^{i-1}\right\|\left(u^{i-1}-u^{k-1}\right).
	 \end{align*}
	 Therefore, taking norms, we get 
	 \begin{align*}
	 \sum_{i=k}^{l}\left\|s^{i-1}\right\|\leq \left\|x^l-x^{k-1}\right\|	+\sum_{i=k}^{l}\left\|s^{i-1}\right\|\left\|u^{i-1}-u^{k-1}\right\|. 
	 \end{align*}
	 According to the Assumption \ref{assumption 2}, we have that the sequence $\left\{x^k\right\}$ is bounded. So, there exists a real number $M>0$ such that $\left\|x^k\right\|\leq M$ for all $k\geq1$. Therefore, we get 
	 \begin{align}\label{convergence theorem proof inequality 1}
	 	\sum_{i=k}^{l}\left\|s^{i-1}\right\|\leq 2M	+\sum_{i=k}^{l}\left\|s^{i-1}\right\|\left\|u^{i-1}-u^{k-1}\right\|.  
	 \end{align}
	 Let $\lambda>0$ be given and we define $\Delta:=\lceil \tfrac{8M}{\lambda}\rceil$, where $\lceil \cdot\rceil$ is the ceiling operator, i.e., $\lceil \tfrac{8M}{\lambda}\rceil$ returns the smallest integer that is greater than or equal to $ \tfrac{8M}{\lambda}$.
	 From Lemma \ref{convergence proof lemma1}, we can find an index $k_0$ such that 
	 \begin{align}\label{convergence theorem proof inequality 2}
	 \underset{k\geq k_0}{\mathlarger\sum}\left\|u^k-u^{k-1}\right\|^2\leq \tfrac{1}{4\Delta}.	 
	 \end{align}
	 With this $\Delta$ and $k_0$, Lemma \ref{convergence proof lemma2} gives an index $k\geq k_0$ such that 
	 \begin{align}\label{convergence theorem proof inequality 3}
	 	\left|\mathcal{K}_{k,\Delta}^\lambda\right|>\tfrac{\Delta}{2}. 
	 \end{align}
	  For any index $i\in\left[k,k+\Delta-1\right]$, using the Cauchy-Schwarz inequality and \eqref{convergence theorem proof inequality 2}, we get 
	  \begin{align*}
	  	\left\|u^{i-1}-u^{k-1}\right\|\leq \sum_{j=k}^{i-1}\left\|u^j-u^{j-1}\right\|\leq \sqrt{i-k}\sqrt{\sum_{j=k}^{i-1}\left\|u^j-u^{j-1}\right\|^2}\leq \sqrt{\Delta}\sqrt{\tfrac{1}{4\Delta}}=\tfrac{1}{2}.
	  \end{align*}
	  Consequently, from \eqref{convergence theorem proof inequality 1}, we get 
	   \begin{align*}
	  	\sum_{i=k}^{l}\left\|s^{i-1}\right\|\leq 2M	+\tfrac{1}{2}\sum_{i=k}^{l}\left\|s^{i-1}\right\|.  
	  \end{align*}
	  For $l=k+\Delta-1$, using \eqref{convergence theorem proof inequality 3}, we have 
	  \[2M\geq \tfrac{1}{2}\sum_{i=k}^{k+\Delta-1}\left\|s^{i-1}\right\|>\tfrac{\lambda}{2}\left|\mathcal{K}_{k,\Delta}^\lambda\right|>\tfrac{\lambda\Delta}{4}.\]
	  Therefore, we get \[\Delta<\tfrac{8M}{\lambda},\]
	  which contradicts to the definition $\Delta$, and this completes the proof.
\end{proof}
\medskip
Next, we study the global convergence of Algorithm \ref{Algorithm} with corresponding to PRP, HS, and LS variants of the NCGM.
\medskip
\begin{theorem}\label{PRP,HS,LS convergence theorem}
\normalfont
Suppose that Assumptions \ref{assumption 1} and \ref{assumption 2} hold. Consider the Algorithm \ref{Algorithm} with \[\beta_k:=\max\left\{\beta_k^{PRP},0\right\}\text{ or }\beta_k:=\max\left\{\beta_k^{HS},0\right\}\text{ or } \beta_k:=\max\left\{\beta_k^{LS},0\right\}.\] If $t_k$ satisfies the standard Wolfe conditions \eqref{Standard Wolfe condition in algorithm}, and $d^k$ is a sufficient descent direction of $G$ at $x^k$, then $\underset{k\to\infty}{\liminf}\left\|v\left(x^k\right)\right\|=0.$	 
\end{theorem}
\medskip
\begin{proof}
	Assume that \[	0<\gamma\leq \left\|v\left(x^k\right)\right\|\leq \bar{\gamma} \text{ for all } k\geq0.\]
	Then, 
	\begin{equation}\label{proof of PRP,HS,LS convergence eq 1}
		\Psi\circ\Upsilon_{x^k}\left(v\left(x^k\right)\right)<-\tfrac{1}{2}\left\|v\left(x^k\right)\right\|^2\leq -\tfrac{1}{2}\gamma^2.
	\end{equation}
	On the other hand, for all $i\in\Lambda$, using Cauchy-Schwarz inequality, we get 
	\begin{align*}
		\Psi\circ\Upsilon_{x^k}\left(v\left(x^k\right)\right)\geq & \tfrac{1}{2}\left(\underline{\nabla_{gH} G_i}\left(x^k\right)+\overline{\nabla_{gH} G_i}\left(x^k\right)\right)^\top v\left(x^k\right)\\
		&+\tfrac{1}{2}\left({\overline{\nabla_{gH} G_i}\left(x^k\right)-\underline{\nabla_{gH} G_i}\left(x^k\right)}\right)^\top \left|v\left(x^k\right)\right|\\
		\geq & -\tfrac{1}{2}\left\|\underline{\nabla_{gH} G_i}\left(x^k\right)+\overline{\nabla_{gH} G_i}\left(x^k\right)\right\|\left\| v\left(x^k\right)\right\|\\
		&-\tfrac{1}{2}\left\|{\overline{\nabla_{gH} G_i}\left(x^k\right)-\underline{\nabla_{gH} G_i}\left(x^k\right)}\right\| \left\|v\left(x^k\right)\right\|. 
	\end{align*}
	Therefore, we get 
	\begin{equation}\label{proof of PRP,HS,LS convergence eq 2}
		-	\Psi\circ\Upsilon_{x^k}\left(v\left(x^k\right)\right)\leq \bar{c} \bar{\gamma},
	\end{equation}
	where $\bar{c}$ is satisfying the relation 
	\[\tfrac{1}{2}\left\|\underline{\nabla_{gH} G_i}\left(x^k\right)+\overline{\nabla_{gH} G_i}\left(x^k\right)\right\|
	+\tfrac{1}{2}\left\|{\overline{\nabla_{gH} G_i}\left(x^k\right)-\underline{\nabla_{gH} G_i}\left(x^k\right)}\right\| \leq \bar{c}\]
	for all $i\in\Lambda$ and for all $k\geq0$.
	Note that there exists an $i_0\in\Lambda$ such that
	\begin{align*}
		 \left|\Psi\circ\Upsilon_{x^{k-1}}\left(v\left(x^k\right)\right) \right|=&\biggl| \tfrac{1}{2}\left(\underline{\nabla_{gH} G_{i_0}}\left(x^{k-1}\right)+\overline{\nabla_{gH} G_{i_0}}\left(x^{k-1}\right)\right)^\top v\left(x^k\right)\\
		 &+\tfrac{1}{2}\left({\overline{\nabla_{gH} G_{i_0}}\left(x^{k-1}\right)-\underline{\nabla_{gH} G_{i_0}}\left(x^{k-1}\right)}\right)^\top \left|v\left(x^k\right)\right|\biggr|\\
		 \leq &\tfrac{1}{2}\left\|\underline{\nabla_{gH} G_{i_0}}\left(x^{k-1}\right)+\overline{\nabla_{gH} G_{i_0}}\left(x^{k-1}\right)\right\|\left\| v\left(x^k\right)\right\|\\
		 & + \tfrac{1}{2}\left\|{\overline{\nabla_{gH} G_{i_0}}\left(x^{k-1}\right)-\underline{\nabla_{gH} G_{i_0}}\left(x^{k-1}\right)}\right\| \left\|v\left(x^k\right)\right\|.
	\end{align*}
	Therefore, we get 
	\begin{equation}\label{proof of PRP,HS,LS convergence eq 3}
	\left|\Psi\circ\Upsilon_{x^{k-1}}\left(v\left(x^k\right)\right) \right|\leq \bar{c}\bar{\gamma}.	 
	\end{equation}
	Combining \eqref{proof of PRP,HS,LS convergence eq 1} and \eqref{proof of PRP,HS,LS convergence eq 2}, we get 
	\begin{equation}\label{proof of PRP,HS,LS convergence eq 4}
		\tfrac{1}{2}\gamma^2<-\Psi\circ\Upsilon_{x^k}\left(v\left(x^k\right)\right)\leq\bar{c}\bar{\gamma}.
	\end{equation}
		We have 
	\begin{align*}
		& -\Psi\circ\Upsilon_{x^{k}}\left(v\left(x^k\right)\right)+\Psi\circ\Upsilon_{x^{k-1}}\left(v\left(x^k\right)\right)\\
		=&-\underset{i\in\Lambda}{\max}~\tfrac{1}{2}\left[\left(\underline{\nabla_{gH}G_i}\left(x^{k}\right)+\overline{\nabla_{gH}G_i}\left(x^{k}\right)\right)^\top v\left(x^k\right)+\left|\overline{\nabla_{gH}G_i}\left(x^{k}\right)-\underline{\nabla_{gH}G_i}\left(x^{k}\right)\right|^\top \left|v\left(x^k\right)\right|\right]\\
		&+\underset{i\in\Lambda}{\max}~\tfrac{1}{2}\left[\left(\underline{\nabla_{gH}G_i}\left(x^{k-1}\right)+\overline{\nabla_{gH}G_i}\left(x^{k-1}\right)\right)^\top v\left(x^k\right)+\left|\overline{\nabla_{gH}G_i}\left(x^{k-1}\right)-\underline{\nabla_{gH}G_i}\left(x^{k-1}\right)\right|^\top \left|v\left(x^k\right)\right|\right]\\
		\leq &\underset{i\in\Lambda}{\max}~\tfrac{1}{2}\biggl[\left(\underline{\nabla_{gH}G_i}\left(x^{k-1}\right)+\overline{\nabla_{gH}G_i}\left(x^{k-1}\right)\right)^\top v\left(x^k\right)+\left|\overline{\nabla_{gH}G_i}\left(x^{k-1}\right)-\underline{\nabla_{gH}G_i}\left(x^{k-1}\right)\right|^\top \left|v\left(x^k\right)\right|\\
		&\hspace{2cm}-\left(\underline{\nabla_{gH}G_i}\left(x^{k}\right)+\overline{\nabla_{gH}G_i}\left(x^{k}\right)\right)^\top v\left(x^k\right)-\left|\overline{\nabla_{gH}G_i}\left(x^{k}\right)-\underline{\nabla_{gH}G_i}\left(x^{k}\right)\right|^\top \left|v\left(x^k\right)\right|\biggr]\\
		\leq &\underset{i\in\Lambda}{\max}~\tfrac{1}{2}\biggl[\left(\underline{\nabla_{gH}G_i}\left(x^{k-1}\right)-\underline{\nabla_{gH}G_i}\left(x^{k}\right)+\overline{\nabla_{gH}G_i}\left(x^{k-1}\right)-\overline{\nabla_{gH}G_i}\left(x^{k}\right)\right)^\top v\left(x^k\right)\\
		&\hspace{2cm}+\left|\overline{\nabla_{gH}G_i}\left(x^{k-1}\right)-\overline{\nabla_{gH}G_i}\left(x^{k}\right)+\underline{\nabla_{gH}G_i}\left(x^{k}\right)-\underline{\nabla_{gH}G_i}\left(x^{k-1}\right)\right|^\top \left|v\left(x^k\right)\right|\biggr]\\
		\leq &\underset{i\in\Lambda}{\max}~\tfrac{1}{2}\biggl[\left\|\underline{\nabla_{gH}G_i}\left(x^{k-1}\right)-\underline{\nabla_{gH}G_i}\left(x^{k}\right)+\overline{\nabla_{gH}G_i}\left(x^{k-1}\right)-\overline{\nabla_{gH}G_i}\left(x^{k}\right)\right\| \left\|v\left(x^k\right)\right\|\\
		&\hspace{2cm}+\left\|\overline{\nabla_{gH}G_i}\left(x^{k-1}\right)-\overline{\nabla_{gH}G_i}\left(x^{k}\right)+\underline{\nabla_{gH}G_i}\left(x^{k}\right)-\underline{\nabla_{gH}G_i}\left(x^{k-1}\right)\right\| \left\|v\left(x^k\right)\right\|\biggr]\\
		\leq &\underset{i\in\Lambda}{\max}~\tfrac{1}{2}\left[2\left(\left\|\underline{\nabla_{gH}G_i}\left(x^{k-1}\right)-\underline{\nabla_{gH}G_i}\left(x^{k}\right)\right\|+\left\|\overline{\nabla_{gH}G_i}\left(x^{k-1}\right)-\overline{\nabla_{gH}G_i}\left(x^{k}\right)\right\|\right) \left\|v\left(x^k\right)\right\|\right]\\
		=&\underset{i\in\Lambda}{\max}~\left(\left\|\underline{\nabla_{gH}G_i}\left(x^{k-1}\right)-\underline{\nabla_{gH}G_i}\left(x^{k}\right)\right\|+\left\|\overline{\nabla_{gH}G_i}\left(x^{k-1}\right)-\overline{\nabla_{gH}G_i}\left(x^{k}\right)\right\|\right) \left\|v\left(x^k\right)\right\|.
	\end{align*}
	Therefore, using Assumption \eqref{assumption 1}, we obtain 
	\begin{align*}
		\left|-\Psi\circ\Upsilon_{x^{k}}\left(v\left(x^k\right)\right)+\Psi\circ\Upsilon_{x^{k-1}}\left(v\left(x^k\right)\right)\right|\leq \underset{i\in\Lambda}{\max}~ 2L_i \left\|x^k-x^{k-1}\right\| \left\|v\left(x^k\right)\right\|=L\left\|s^{k-1}\right\| \left\|v\left(x^k\right)\right\|,
	\end{align*}
	 where $L:=\underset{i\in\Lambda}{\max}~ 2L_i.$ Therefore, for $\left\|s^{k-1}\right\|\leq\lambda$, we get 
	 \begin{equation}\label{proof of PRP,HS,LS convergence eq 5}
	 	\left|-\Psi\circ\Upsilon_{x^{k}}\left(v\left(x^k\right)\right)+\Psi\circ\Upsilon_{x^{k-1}}\left(v\left(x^k\right)\right)\right|\leq L\lambda\bar{\gamma}. 
	 \end{equation}
	 For the PRP method, we define $b:=\tfrac{4\bar{c}\bar{\gamma}}{\gamma^2}$ and $\lambda:=\tfrac{\gamma^2}{4L\bar{\gamma}b}$. Using \eqref{proof of PRP,HS,LS convergence eq 3} and \eqref{proof of PRP,HS,LS convergence eq 4}, we get 
	 \[\left|\beta_k^{PRP}\right|\leq \tfrac{-\Psi\circ\Upsilon_{x^k}\left(v\left(x^k\right)\right)+\left|\Psi\circ\Upsilon_{x^{k-1}}\left(v\left(x^k\right)\right)\right|}{-\Psi\circ\Upsilon_{x^{k-1}}\left(v\left(x^{k-1}\right)\right)}\leq\tfrac{4\bar{c}\bar{\gamma}}{\gamma^2}=b,\]
	 and for $\left\|s^{k-1}\right\|\leq\lambda$, using \eqref{proof of PRP,HS,LS convergence eq 4} and \eqref{proof of PRP,HS,LS convergence eq 5}, we get that 
	 \[\left|\beta_k^{PRP}\right|\leq \tfrac{2L\lambda\bar{\gamma}}{\gamma^2}=\tfrac{1}{2b},\]
	 which shows that the PRP method has Property \ref{property}. Therefore, from Theorem \ref{convergence theorem}, it follows that \[\underset{k\to\infty}{\liminf}\left\|v\left(x^k\right)\right\|=0.\]	
	 
	 Next, we consider the HS method. Using the standard Wolfe conditions \eqref{Standard Wolfe condition in algorithm}, the sufficient descent condition \eqref{sufficient descent condition}, and \eqref{proof of PRP,HS,LS convergence eq 4}, we obtain that 
	 \begin{align*}
	 	 \Psi\circ\Upsilon_{x^{k}}\left(d^{k-1}\right)-\Psi\circ\Upsilon_{x^{k-1}}\left(d^{k-1}\right)&\geq \sigma\Psi\circ\Upsilon_{x^{k-1}}\left(d^{k-1}\right)-\Psi\circ\Upsilon_{x^{k-1}}\left(d^{k-1}\right)\\
	 	 &=-\left(1-\sigma\right)\Psi\circ\Upsilon_{x^{k-1}}\left(d^{k-1}\right)\\
	 	 &\geq -c\left(1-\sigma\right)\Psi\circ\Upsilon_{x^{k-1}}\left(v\left(x^{k-1}\right)\right)\\
	 	 &\geq \tfrac{c\left(1-\sigma\right)\gamma^2}{2}>0.
	 \end{align*}
	 Therefore, we have 
	 \begin{equation}\label{proof of PRP,HS,LS convergence eq 6}
	 	\Psi\circ\Upsilon_{x^{k}}\left(d^{k-1}\right)-\Psi\circ\Upsilon_{x^{k-1}}\left(d^{k-1}\right)\geq \tfrac{c\left(1-\sigma\right)\gamma^2}{2}>0. 
	 \end{equation}
	 Let us now define $b:=\tfrac{4\bar{c}\bar{\gamma}}{c\left(1-\sigma\right)\gamma^2}$ and $\lambda:=\tfrac{c\left(1-\sigma\right)\gamma^2}{4L\bar{\gamma}b}$.
	 Using \eqref{proof of PRP,HS,LS convergence eq 3}, \eqref{proof of PRP,HS,LS convergence eq 4}, and \eqref{proof of PRP,HS,LS convergence eq 6}, we get
	 \[\left|\beta_k^{HS}\right|\leq \tfrac{-\Psi\circ\Upsilon_{x^k}\left(v\left(x^k\right)\right)+\left|\Psi\circ\Upsilon_{x^{k-1}}\left(v\left(x^k\right)\right)\right|}{\Psi\circ\Upsilon_{x^{k}}\left(d^{k-1}\right)-\Psi\circ\Upsilon_{x^{k-1}}\left(d^{k-1}\right)}\leq\tfrac{4\bar{c}\bar{\gamma}}{c\left(1-\sigma\right)\gamma^2}=b,\]
	 and for $\left\|s^{k-1}\right\|\leq\lambda$, using \eqref{proof of PRP,HS,LS convergence eq 5} and \eqref{proof of PRP,HS,LS convergence eq 6}, we get that 
	 \[\left|\beta_k^{HS}\right|\leq \tfrac{2L\lambda\bar{\gamma}}{c\left(1-\sigma\right)\gamma^2}=\tfrac{1}{2b},\]
	 which shows that the HS method has Property \ref{property}. Therefore, from Theorem \ref{convergence theorem}, it follows that \[\underset{k\to\infty}{\liminf}\left\|v\left(x^k\right)\right\|=0.\]	
	 
	 Let us consider the LS method. Using the sufficient descent condition \eqref{sufficient descent condition} and \eqref{proof of PRP,HS,LS convergence eq 4}, we obtain that 
	 \begin{equation}\label{proof of PRP,HS,LS convergence eq 7}
	 	-\Psi\circ\Upsilon_{x^{k-1}}\left(d^{k-1}\right)\geq -c\Psi\circ\Upsilon_{x^{k-1}}\left(v\left(x^{k-1}\right)\right)\geq \tfrac{c\gamma^2}{2}>0.
	 \end{equation}
	  Let us now define $b:=\tfrac{4\bar{c}\bar{\gamma}}{c\gamma^2}$ and $\lambda:=\tfrac{c\gamma^2}{4L\bar{\gamma}b}$.
	 Using \eqref{proof of PRP,HS,LS convergence eq 3}, \eqref{proof of PRP,HS,LS convergence eq 4}, and \eqref{proof of PRP,HS,LS convergence eq 7}, we get
	 \[\left|\beta_k^{LS}\right|\leq \tfrac{-\Psi\circ\Upsilon_{x^k}\left(v\left(x^k\right)\right)+\left|\Psi\circ\Upsilon_{x^{k-1}}\left(v\left(x^k\right)\right)\right|}{-\Psi\circ\Upsilon_{x^{k-1}}\left(d^{k-1}\right)}\leq\tfrac{4\bar{c}\bar{\gamma}}{c\gamma^2}=b,\]
	 and for $\left\|s^{k-1}\right\|\leq\lambda$, using \eqref{proof of PRP,HS,LS convergence eq 5} and \eqref{proof of PRP,HS,LS convergence eq 7}, we get that 
	 \[\left|\beta_k^{LS}\right|\leq \tfrac{2L\lambda\bar{\gamma}}{c\gamma^2}=\tfrac{1}{2b},\]
	 which shows that the LS method has Property \ref{property}. Therefore, from Theorem \ref{convergence theorem}, it follows that \[\underset{k\to\infty}{\liminf}\left\|v\left(x^k\right)\right\|=0.\]	
	 Thus, the proof is concluded.
\end{proof}
\section{Numerical Experiments}\label{Numerical Experiments}
In this section, we make a report on the computational performance of Algorithm \ref{Algorithm} applied on a collection of standard test problems taken from \cite{mondal2025steepest}. All computations were implemented in MATLAB 2023a and executed on a machine with an Intel(R) Core(TM) i5-1035G1 (10th generation) processor operating between 1.00 GHz and 1.19 GHz, supported by 8 GB RAM.

The MATLAB implementation of Algorithm \ref{Algorithm} is carried out under the following configuration:
\begin{itemize}
	\item The starting point $x^0$ is randomly generated within the feasible domain prescribed for each problem, consistent with the bounds reported in \cite{mondal2025steepest}. The sampling is performed using the MATLAB command ``rand".
	
	\item The Wolfe line search parameters are taken across all experiments as $\rho = 0.001$ and $\sigma = 0.1$. At every iteration $k$, the step size $t_k$ is computed so that the standard Wolfe conditions in \eqref{Standard Wolfe condition in algorithm} are satisfied.
	
	\item At each iteration, the quantities $v(x^k)$ and $\xi(x^k)$ are obtained by solving the quadratic subproblem given in \eqref{equivalent constrained problem} using the ``quadprog" solver available in MATLAB.
	
	\item The stopping rule for Algorithm \ref{Algorithm} is based on the condition $\xi(x^k) > -\epsilon$. In this termination condition, the acceptance or the tolerance parameter is fixed at $\epsilon = 10^{-6}$.
	
	\item The parameter $\beta_k$ is selected differently depending on the variant: for the PRP, HS, and LS methods, we use $\max\{0,\beta_k^{PRP}\}$, $\max\{0,\beta_k^{HS}\}$, and $\max\{0,\beta_k^{LS}\}$, respectively.
\end{itemize}

With the above setup, Algorithm \ref{Algorithm} is applied to all test problems from \cite{mondal2025steepest}. To evaluate its efficiency and make its comparison with existing approaches, including the SD method \cite{mondal2025steepest} and other NCGMs namely FR, CD, DY, and mDY \cite{mondal2026nonlinear}, we record the number of iterations and computational time (in seconds). These results are summarized in Table \ref{performance table on iteration} and Table \ref{performance table on CPU}.

For each test problem, 100 randomly generated initial points are considered. Based on these runs, we report the triplet $(\min, \text{mean}, \max)$ for both iteration counts and computational time.

To further compare performance, we adopt the Dolan--Mor{\'e} performance profile framework \cite{dolan2002benchmarking}. Let $\mathcal{S}$ denote the set of solvers and $\mathcal{P}$ the collection of test problems. Let $N_s$ and $N_p$ be the number of solvers and problems, respectively. We use the mean iteration count and mean computational time as performance indicators. For every problem $p \in \mathcal{P}$ and solver $s \in \mathcal{S}$, define
\[
I_{p,s} := \text{mean iteration count required by solver } s \text{ to solve problem } p, \mbox{ and }
\]
\[
T_{p,s} := \text{mean CPU time required by solver } s \text{ to solve problem } p.
\]

To assess relative efficiency, we introduce the performance ratios
\[
R_{p,s}^I := \frac{I_{p,s}}{\min\{I_{p,s} : s \in \mathcal{S}\}} 
\text{ and }
R_{p,s}^T := \frac{T_{p,s}}{\min\{T_{p,s} : s \in \mathcal{S}\}}.
\]

The performance profiles associated with iteration counts and CPU time are defined as functions $F_I, F_T : \mathbb{R} \to [0,1]$ by
\[F_I\left(z\right):=\tfrac{1}{N_p}\text{ size }\left\{p\in{\mathcal{P}}:R_{p,s}^I\leq z\right\} \text{ and } F_T\left(z\right):=\tfrac{1}{N_p}\text{ size }\left\{p\in{\mathcal{P}}:R_{p,s}^T\leq z\right\}\text{ for every } z\in {\mathbb{R}}.\]

These functions represent the probability that a given solver $s \in \mathcal{S}$ performs within a factor $z$ of the best solver, corresponding to iteration count and CPU time, respectively. In other words, $F_I$ and $F_T$ correspond to the cumulative distribution functions of the ratios $R_{p,s}^I$ and $R_{p,s}^T$. The resulting performance profiles are depicted in Figure \ref{figure:performance profile}.

		\begin{figure}[H]
		\begin{subfigure}[t]{0.42\textwidth}
			\includegraphics[width=\linewidth]{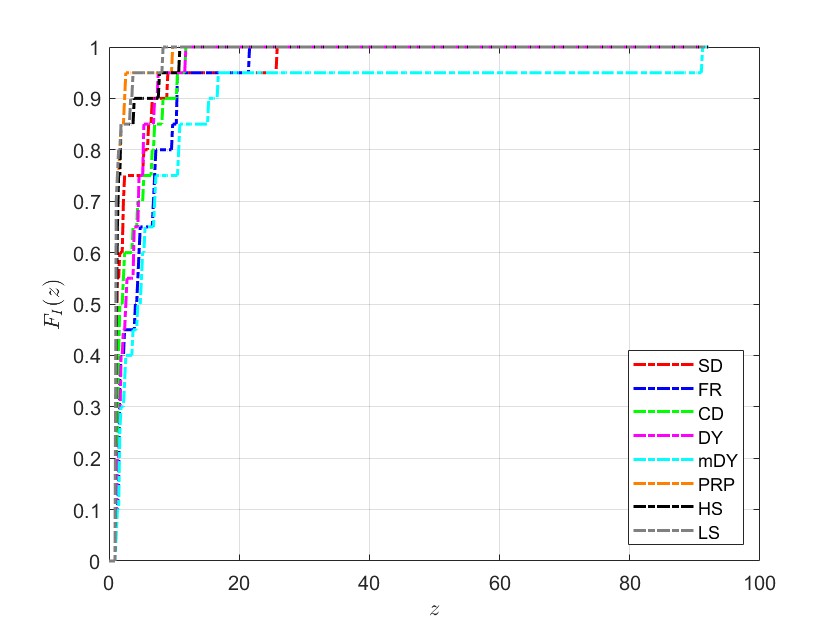}
			\caption{Performance profile with respect to mean iteration counts}
			\label{iteration}
		\end{subfigure}\hfill
		\begin{subfigure}[t]{0.42\textwidth}
			\includegraphics[width=\linewidth]{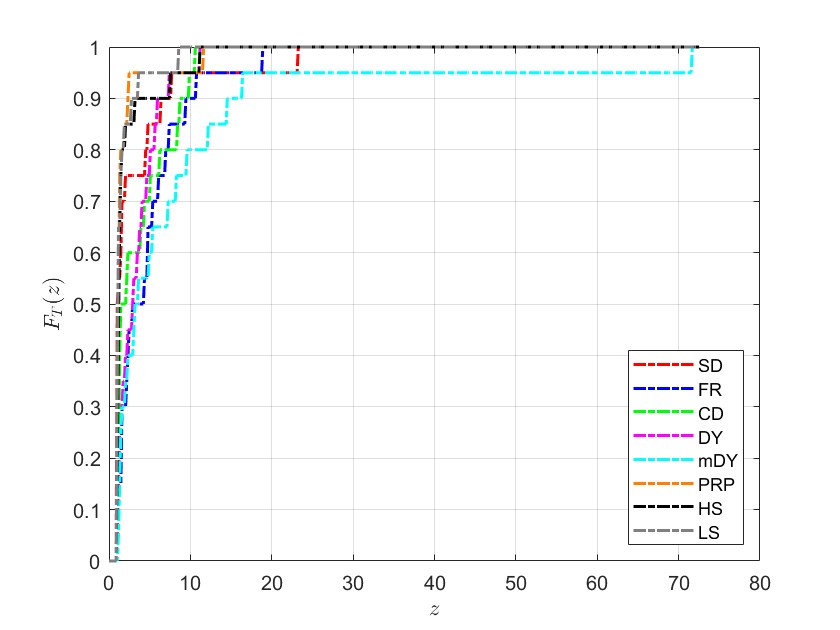}
			\caption{Performance profile with respect to mean computational time}
			\label{CPU time}
		\end{subfigure}
		\caption{Performance profile.}
		\label{figure:performance profile}
	\end{figure}

\clearpage
\begin{landscape}
{\footnotesize 
\centering 
\begin{longtable}{@{} l *{8}{c} @{}}
		\caption{Performance of Algorithm \ref{Algorithm} with respect to iteration count \label{performance table on iteration}} \\
		
		\toprule
		\multirow{2}{*}{\text{Problem}} &
		\text{SD} & \text{FR} & \text{CD} & \text{DY} & \text{mDY} & \text{PRP} & \text{HS} & \text{LS} \\
		& {\footnotesize (min, mean, max)} & {\footnotesize (min, mean, max)} & {\footnotesize (min, mean, max)} & {\footnotesize (min, mean, max)} & {\footnotesize (min, mean, max)} & {\footnotesize (min, mean, max)} & {\footnotesize (min, mean, max)} & {\footnotesize (min, mean, max)} \\
		\midrule
		\endfirsthead
		
		\toprule
		\multirow{2}{*}{\text{Problem}} &
		\text{SD} & \text{FR} & \text{CD} & \text{DY} & \text{mDY} & \text{PRP} & \text{HS} & \text{LS} \\
		& {\footnotesize (min, mean, max)} & {\footnotesize (min, mean, max)} & {\footnotesize (min, mean, max)} & {\footnotesize (min, mean, max)} & {\footnotesize (min, mean, max)} & {\footnotesize (min, mean, max)} & {\footnotesize (min, mean, max)} & {\footnotesize (min, mean, max)} \\
		\midrule
		\endhead
		
		\midrule \multicolumn{9}{r}{\footnotesize Continued on next page} \\
		\endfoot
		
		\bottomrule
		\endlastfoot
		
		I-BK1 & (0, 2.49, 5) & (0, 3.07, 13) & (0, 3.21, 9) & (0, 2.83, 7) & (0, 3.45, 11) & (0, 2.67, 5) & (0, 2.63, 5) & (0, 2.24, 5)\\
		I-VU2 & (0, 27.74, 69) & (0, 228.32, 667)
		 & (0, 192.80, 571) & (0, 125.57, 313) & (0, 394.51, 1206) & (0, 23.70, 63) & (0, 25.67, 65) & (0, 24.06, 65) \\
		I-CH &  (1, 5.87, 11) & (1, 28.01, 106)
		 & (1, 13.07, 32) & (1, 10.75, 22) & (1, 21.20, 75)  & (1, 5.97, 9) & (1, 5.99, 9) & (1, 6.04, 9)\\
		I-FON &   (0, 10.75, 45) & (0, 112.84, 493)
		 & (0, 46.24, 329) & (0, 48.80, 199) & (0, 114.55, 748) & (0, 12.22, 54) & (0, 10.97, 57) & (0, 11.02, 50)\\
		 I-KW2 &   (0, 19.02, 67) & (0, 82.30, 648)
		  & (0, 124.78, 1621) & (0, 86.09, 609) & (0, 135.84, 820) & (0, 32.22, 150) & (0, 34.43, 232) & (0, 32.72, 141)\\
		 I-Far1 &   (0, 4.15, 57) & (0, 18.92, 189) & (0, 21.46, 631) & (0, 30.97, 352) & (0, 20.21, 259) & (0, 9.71, 163) & (0, 5.22, 58) & (0, 5.49, 216)\\
		 I-Hil1 &   (0, 10.86, 111) & (0, 46.01, 340) & (0, 47.74, 480) & (0, 25.91, 164) & (0, 73.22, 684) & (0, 6.83, 60) & (0, 9.58, 70) & (0, 8.21, 67)\\
		 I-PNR &    (0, 6.43, 22) & (0, 43.94, 126) & (0, 23.65, 167) & (0, 17.73, 86) & (0, 44.90, 146) & (0, 7.86, 23) & (0, 6.91, 16)
		  & (0, 7.09, 15)\\
		 I-Deb &    (0, 224.71, 20007) & (0, 13.16, 85)
		  & (0, 8.74, 79) & (0, 15.62, 105) & (0, 10.39, 92) & (0, 21.27, 1712) & (0, 66.40, 1563) & (0, 31.35, 2750)\\
		 I-SD &   (0, 3.64, 8) & (0, 6.00, 20)
		 & (0, 4.48, 14) & (0, 4.09, 11) & (0, 5.02, 15) & (0, 3.48, 7) & (0, 3.30, 7) & (0, 3.34, 6)\\
    	I-IKK1 &   (0, 6.28, 349) & (0, 21.56, 880)
		 & (0, 35.81, 1251) & (0, 16.06, 488) & (0, 277.85, 8227) & (0, 3.05, 104) & (0, 11.56, 176) & (0, 9.75, 417)\\
		 I-VFM1 &    (0, 1.52, 4) & (0, 1.97, 6)
		 & (0, 1.65, 6) & (0, 1.46, 5) & (0, 1.88, 6) & (0, 1.58, 5) & (0, 1.50, 4) & (0, 1.24, 4)\\
		 I-MHHM2 &    (0, 3.37, 6) & (1, 8.06, 22)
		 & (1, 7.13, 14) & (0, 6.11, 12) & (0, 7.92, 21) & (1, 3.57, 6) & (1, 3.46, 6) & (0, 3.39, 6)\\
I-Viennet & (0, 1.88, 100) & (0, 3.23, 103) & (0, 1.29, 37) & (0, 1.65, 36) & (0, 4.09, 88) & (0, 1.25, 15) & (0, 1.15, 22) & (0, 0.80, 12)\\
I-AP1 &    (0, 93.98, 1569) & (0, 343.20, 3408)
& (0, 167.00, 1571) & (1, 111.50, 1095) & (0, 28.40, 237) & (0, 18.30, 154) & (0, 16.20, 146) & (0, 15.90, 138)\\
	I-MOP7 &    (25, 67.78, 125) & (2, 7.56, 15)
& (2, 7.99, 14) & (2, 8.08, 15) & (2, 8.14, 15) & (2, 7.63, 11)
 & (3, 7.53, 11) & (4, 7.55, 11)\\
I-VFM2  &    (0, 17.82, 34) & (0, 35.47, 2703)
& (0, 5.50, 87) & (0, 5.04, 58) & (0, 14.74, 351)
& (0, 3.40, 25) & (0, 3.62, 36)
 & (0, 3.51, 40)\\
I-TR1 &    (8, 8.64, 10) & (4, 4.00, 4) & (5, 5.00, 5) & (8, 9.60, 12) & (8, 10.03, 13) & (4, 4.00, 4) & (6, 6.68, 7) & (4, 4.00, 4)\\
I-AP4 &    (0, 141.71, 2301) & (0, 34.88, 301)
& (0, 21.56, 209) & (0, 253.04, 5989) & (0, 328.60, 7345) & (0, 207.51, 2667) & (1, 233.44, 2307) & (0, 177.67, 2300)\\
I-Comet &    (0, 80.01, 131) & (0, 117.40, 278)
& (0, 85.20, 386) & (0, 294.80, 602) & (0, 425.00, 903) & (0, 85.93, 136) & (0, 89.49, 135) & (0, 77.89, 134)\\
\end{longtable} }    
\end{landscape}

    \clearpage
\begin{landscape}
{\footnotesize     
	\begin{longtable}{@{} l *{8}{c} @{}}
		\caption{Performance of Algorithm \ref{Algorithm} with respect to CPU time (s)\label{performance table on CPU}} \\
		
		\toprule
		\multirow{2}{*}{\text{Problem}} &
		\text{SD} & \text{FR} & \text{CD} & \text{DY} & \text{mDY} & \text{PRP} & \text{HS} & \text{LS} \\
		& {\footnotesize (min, mean, max)} & {\footnotesize (min, mean, max)} & {\footnotesize (min, mean, max)} & {\footnotesize (min, mean, max)} & {\footnotesize (min, mean, max)}  & {\footnotesize (min, mean, max)} & {\footnotesize (min, mean, max)} & {\footnotesize (min, mean, max)} \\
		\midrule
		\endfirsthead
		
		\toprule
		\multirow{2}{*}{\text{Problem}} &
		\text{SD} & \text{FR} & \text{CD} & \text{DY} & \text{mDY} & \text{PRP} & \text{HS} & \text{LS} \\
		& {\footnotesize (min, mean, max)} & {\footnotesize (min, mean, max)} & {\footnotesize (min, mean, max)} & {\footnotesize (min, mean, max)} & {\footnotesize (min, mean, max)} & {\footnotesize (min, mean, max)} & {\footnotesize (min, mean, max)} & {\footnotesize (min, mean, max)} \\
		\midrule
		\endhead
		
		\midrule \multicolumn{9}{r}{\footnotesize Continued on next page} \\
		\endfoot
		
		\bottomrule
		\endlastfoot
		
		I-BK1 &  (0.02, 0.07, 0.12) & (0.02, 0.09, 0.30) & (0.02, 0.09, 0.21) & (0.02, 0.08, 0.17) & (0.02, 0.09, 0.24)  & (0.02, 0.08, 0.17) & (0.02, 0.08, 0.13) & (0.02, 0.07, 0.15)\\
		IVU2  & (0.02, 0.56, 1.36) & (0.02, 5.07, 15.93) & (0.02, 5.27, 16.14) & (0.02, 3.03, 15.59) & (0.02, 8.80, 35.29)  & (0.02, 0.54, 1.37) & (0.02, 0.57, 1.48) & (0.02, 0.54, 1.39)\\
		I-CH &  (0.04, 0.13, 0.28) & (0.04, 0.61, 2.25) & (0.04, 0.30, 0.71) & (0.04, 0.25, 0.50) & (0.04, 0.47, 1.58)  & (0.04, 0.15, 0.23) & (0.04, 0.16, 0.26) & (0.04, 0.15, 0.21)\\
		I-FON &   (0.02, 0.23, 0.89) & (0.02, 2.46, 10.62) & (0.02, 1.00, 6.95) & (0.02, 1.07, 4.27) & (0.02, 2.80, 16.35)  & (0.02, 0.28, 1.15) & (0.02, 0.25, 1.19) & (0.02, 0.26, 1.16)\\
		I-KW2 &   (0.02, 0.39, 1.31) & (0.02, 1.87, 14.49) & (0.02, 3.25, 36.36) & (0.02, 1.93, 13.97) & (0.02, 3.21, 18.18)  & (0.02, 0.72, 3.23) & (0.02, 0.77, 5.04) & (0.02, 0.72, 2.99)\\
		I-Far1 &    (0.02, 0.10, 1.10) & (0.02, 0.43, 4.07) & (0.02, 0.50, 14.19) & (0.02, 0.74, 7.65) & (0.02, 0.49, 5.98)  & (0.02, 0.23, 3.48) & (0.02, 0.14, 1.27) & (0.02, 0.14, 4.80)\\
		I-Hil1 &    (0.02, 0.23, 2.15) & (0.02, 1.02, 7.32) & (0.02, 1.05, 10.43) & (0.02, 0.58, 3.50) & (0.02, 1.61, 14.81)  & (0.02, 0.17, 1.30) & (0.02, 0.23, 1.49) & (0.02, 0.23, 1.48)\\
		I-PNR &   (0.02, 0.14, 0.44) & (0.02, 0.97, 2.76) & (0.02, 0.53, 3.58) & (0.02, 0.40, 1.91) & (0.02, 1.00, 3.10)  & (0.02, 0.19, 0.52) & (0.02, 0.17, 0.37) & (0.02, 0.21, 0.54)\\
		I-Deb &    (0.02, 4.63, 413.67) & (0.02, 0.30, 1.86) & (0.02, 0.20, 1.69) & (0.02, 0.60, 4.95) & (0.02, 0.24, 2.01)  & (0.02, 0.48, 36.80) & (0.02, 1.49, 34.53) & (0.02, 0.71, 60.24)\\
		I-SD &    (0.02, 0.09, 0.17) & (0.02, 0.21, 0.82)
		& (0.02, 0.12, 0.36) & (0.02, 0.11, 0.25) & (0.02, 0.13, 0.38) & (0.02, 0.10, 0.17) & (0.02, 0.09, 0.17) & (0.02, 0.10, 0.18)\\
		I-IKK1 &    (0.02, 0.14, 6.86) & (0.02, 0.48, 18.73)
		& (0.02, 0.78, 26.26) & (0.02, 0.36, 10.39) & (0.02, 6.44, 193.81) & (0.02, 0.09, 2.25) & (0.02, 0.28, 3.86) & (0.02, 0.24, 9.15)\\
		 I-VFM1 &    (0.02, 0.05, 0.13) & (0.02, 0.14, 6.60)
		& (0.02, 0.07, 0.19) & (0.02, 0.05, 0.18) & (0.02, 0.07, 0.21) & (0.02, 0.06, 0.14) & (0.02, 0.05, 0.16) & (0.02, 0.05, 0.12)\\
		I-MHHM2 &    (0.02, 0.09, 0.16) & (0.04, 0.19, 0.49)
		& (0.04, 0.19, 0.52) & (0.02, 0.15, 0.28) & (0.02, 0.18, 0.43) & (0.04, 0.10, 0.18) & (0.04, 0.10, 0.19) & (0.02, 0.10, 0.19)\\
		I-Viennet &    (0.02, 0.06, 1.96) & (0.02, 0.09, 2.11)
		& (0.02, 0.05, 0.73) & (0.02, 0.06, 0.86) & (0.02, 0.12, 2.10) & (0.02, 0.05, 0.36) & (0.02, 0.05, 0.53) & (0.02, 0.04, 0.28)\\
		I-AP1 &    (0.02, 1.81, 30.15) & (0.02, 7.16, 70.86)
		& (0.02, 3.98, 36.89) & (0.04, 2.23, 21.73) & (0.02, 0.58, 4.71) & (0.02, 0.42, 3.38) & (0.02, 0.38, 3.26) & (0.02, 0.38, 3.08)\\
		I-MOP7 &    (0.48, 1.30, 2.43) & (0.06, 0.17, 0.31)
		& (0.06, 0.18, 0.30) & (0.06, 0.18, 0.31) & (0.06, 0.18, 0.31) & (0.06, 0.19, 0.28) & (0.09, 0.19, 0.26) & (0.11, 0.19, 0.27)\\
	I-VFM2 &    (0.02, 0.45, 2.19) & (0.02, 0.73, 53.82)
		& (0.02, 0.13, 1.77) & (0.02, 0.12, 1.20) & (0.02, 0.31, 6.95) & (0.02, 0.10, 0.57) & (0.02, 0.10, 0.81) & (0.02, 0.10, 0.93)\\
		I-TR1 &    (0.02, 0.20, 0.39) & (0.10, 0.10, 0.13)
		&(0.12, 0.13, 0.15) & (0.18, 0.22, 0.27) & (0.18, 0.23, 0.29) & (0.10, 0.12, 0.15) & (0.13, 0.15, 0.18) & (0.09, 0.10, 0.12)\\
		I-AP4 &    (0.02, 2.92, 55.34) & (0.02, 0.72, 6.06)
		& (0.02, 0.46, 4.22) & (0.02, 5.10, 120.28) & (0.02, 6.64, 147.83) & (0.02, 5.31, 78.80) & (0.04, 5.10, 49.97) & (0.02, 3.92, 49.43)\\
		I-Comet &    (0.02, 1.60, 2.94) & (0.02, 2.41, 5.85)
		& (0.02, 1.74, 7.79) & (0.02, 5.97, 12.12) & (0.02, 8.54, 18.11) & (0.02, 1.99, 3.81) & (0.02, 2.05, 3.09) & (0.02, 1.76, 3.13)\\
		
	\end{longtable}}  

\end{landscape}

Based on the results reported in Table \ref{performance table on iteration} and Table \ref{performance table on CPU}, the following conclusions can be drawn:
\begin{itemize}
	\item Within the PRP, HS, and LS variants of the NCGM, the PRP scheme performs most efficiently for the problems I-CH, I-Deb, I-Hil1, I-KW2, I-VU2, I-IKK1, and I-VFM2.
	
	\item The HS variant exhibits superior performance for the problems I-Far1, I-FON, I-PNR, I-SD, and I-MOP7 when compared to the other two variants.
	
	\item The LS variant shows better results for the problems I-BK1, I-AP1, I-AP4, I-Comet, I-MHHM2, and I-Viennet.
\end{itemize}

For each individual test problem, we adopt the variant (among PRP, HS, and LS) that yields the most favorable performance. For instance, in the case of problem I-Hil1, the PRP variant is employed.

\bigskip
\noindent
We randomly select five starting points for biobjective test problems given in \cite{mondal2025steepest} and plot their corresponding optimal objective values, $G(x^\star)$, within the feasible objective space in Figure \ref{figure:biobjective}. The gray shaded area represents the feasible objective space, while each rectangle filled with magenta color indicates the image of an optimal solution $G(x^\star)$. 

To depict the feasible objective space, we randomly generate $5000$ points from the domain of the decision variables for each problem. For every single point $x$, we obtain a rectangular image $G(x)$, and the union of all such images provides an approximation of the feasible objective space, namely,
\[
\text{objective feasible space} := \bigcup_{L \le x \le U} G(x).
\]

In the graphical representation, the small circle filled with black color denotes the center of $G(x^0)$, whereas the small circle filled with blue color indicates the center of $G(x^\star)$. These two centers are joined by line segment with magenta color.

\bigskip
\noindent
Visualizing the entire feasible objective space is not practical for triobjective test problems. Due to this reason, we depict the objective functions $G$ at a randomly generated initial point $x^0$ and at its converging point $x^\star$ for the triobjective test problems in Figure \ref{figure:triobjective}. In this figure, the cube filled with cyan color represents $G(x^0)$, while the cube filled with magenta indicates $G(x^\star)$. In addition, the curve with magenta color represents the trajectory of the centers of the sequence $\{G(x^k)\}$ generated by Algorithm \ref{Algorithm}.
	
	\begin{figure}[htbp]
		\begin{subfigure}[t]{0.33\textwidth}
			\includegraphics[width=\linewidth]{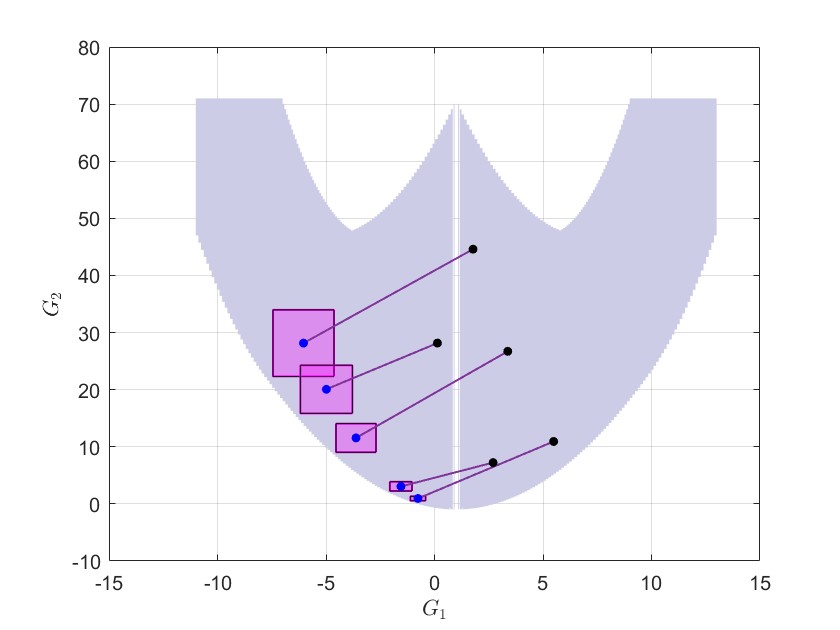}
			\caption{I-VU2}
			\label{fig:I-VU2}
		\end{subfigure}\hfill
		\begin{subfigure}[t]{0.33\textwidth}
			\includegraphics[width=\linewidth]{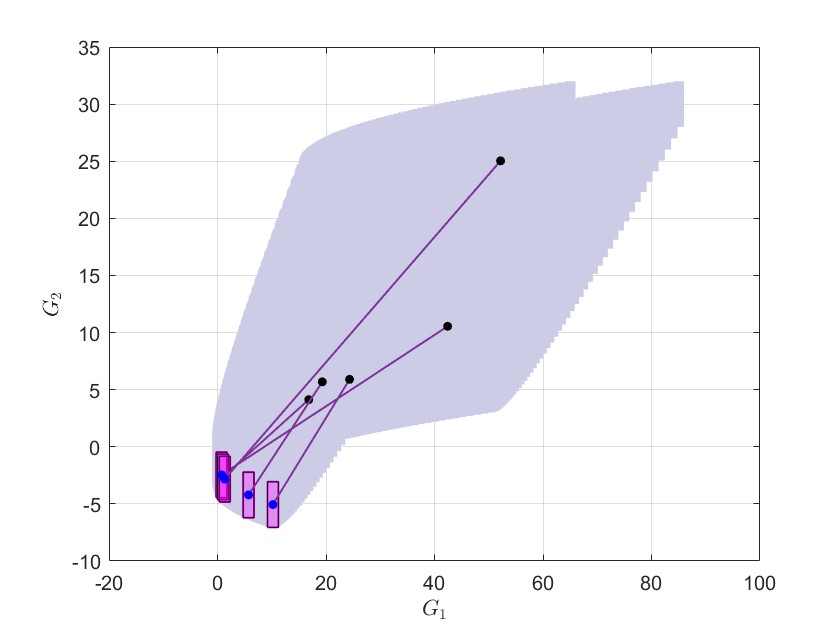}
			\caption{I-CH}
			\label{fig:I-CH}
		\end{subfigure}\hfill
		\begin{subfigure}[t]{0.33\textwidth}
			\includegraphics[width=\linewidth]{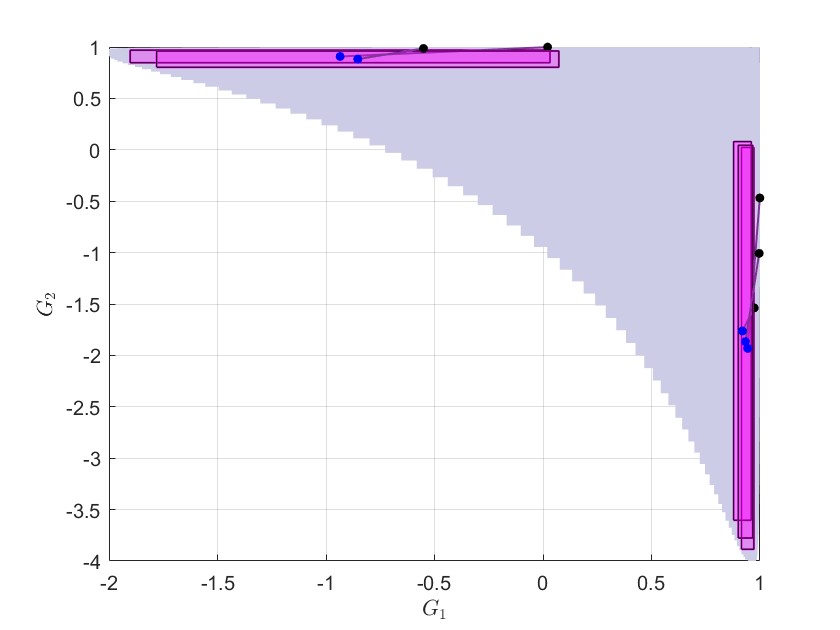}
			\caption{I-FON}
			\label{fig:I-FON}
		\end{subfigure}
		
		\begin{subfigure}[t]{0.33\textwidth}
			\includegraphics[width=\linewidth]{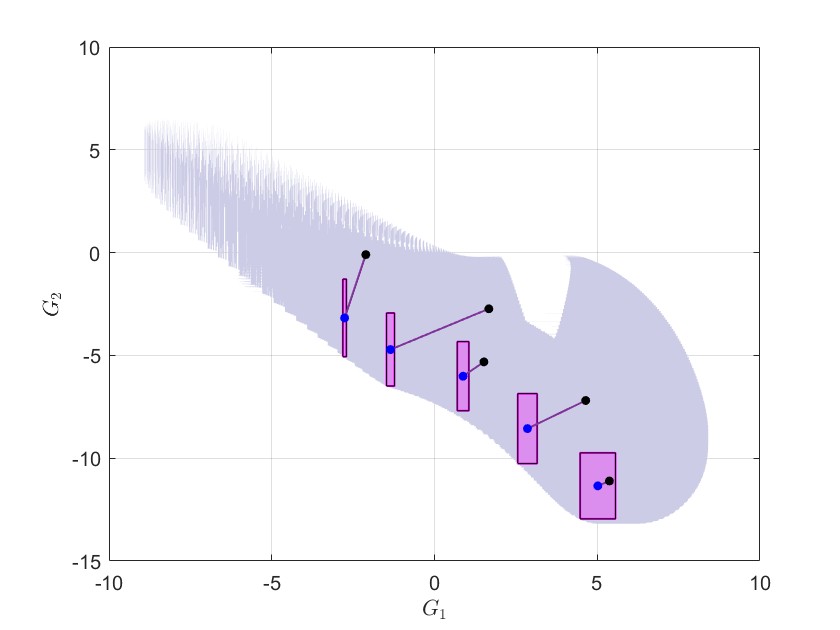}
			\caption{I-KW2}
			\label{fig:I-KW2}
		\end{subfigure}\hfill
		\begin{subfigure}[t]{0.33\textwidth}
			\includegraphics[width=\linewidth]{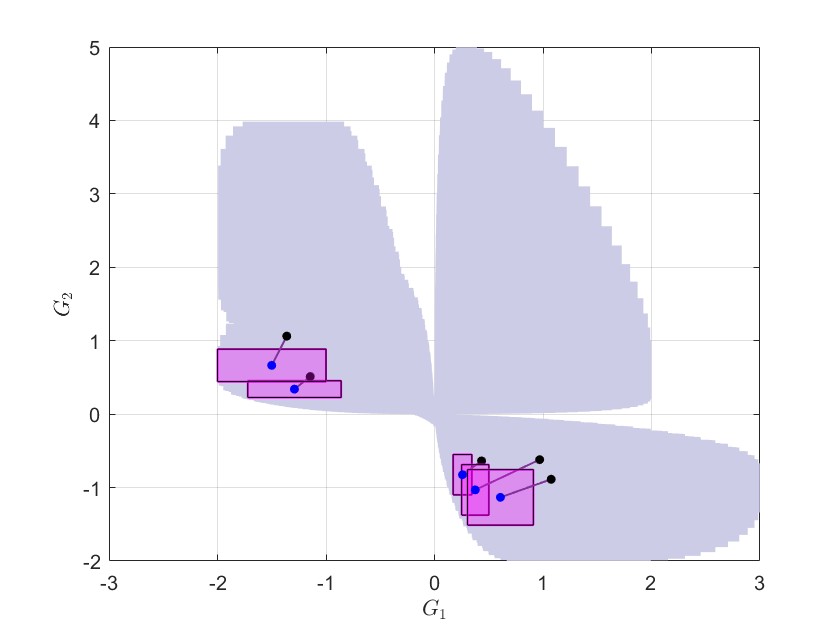}
			\caption{I-Far1}
			\label{fig:I-Far1}
		\end{subfigure}\hfill
		\begin{subfigure}[t]{0.33\textwidth}
			\includegraphics[width=\textwidth]{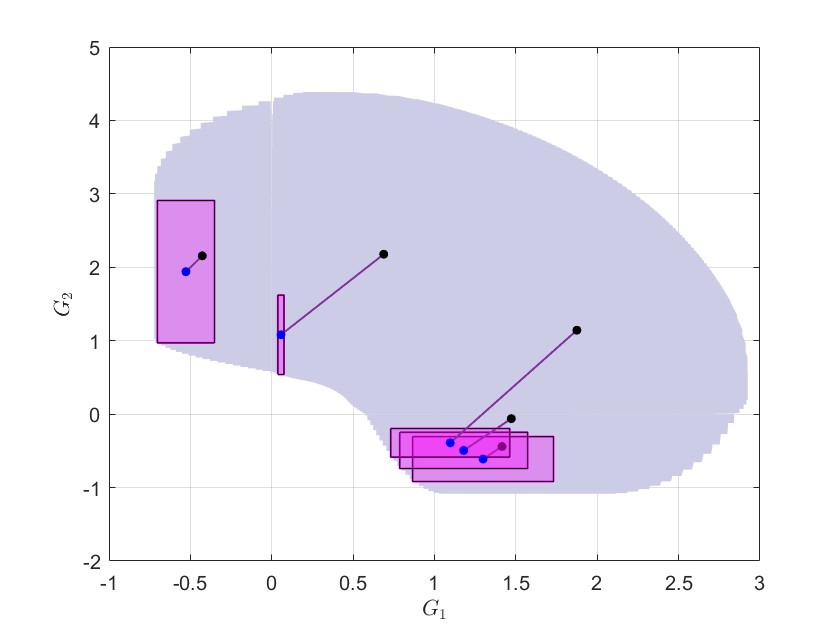}
			\caption{I-Hil1}
			\label{fig:I-Hil1}
		\end{subfigure}
		
		\begin{subfigure}[t]{0.33\textwidth}
			\includegraphics[width=\linewidth]{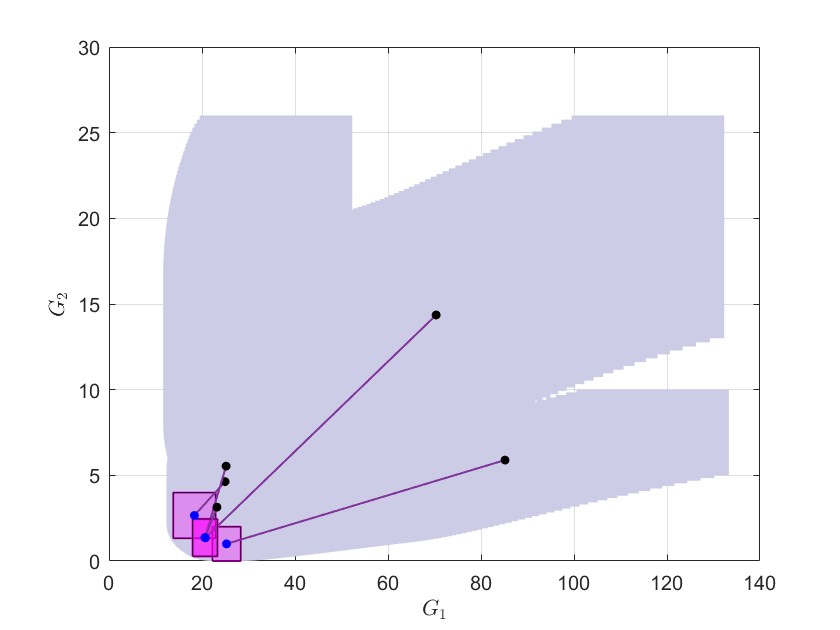}
			\caption{I-PNR}
			\label{fig:I-PNR}
		\end{subfigure}\hfill
		\begin{subfigure}[t]{0.33\textwidth}
			\includegraphics[width=\linewidth]{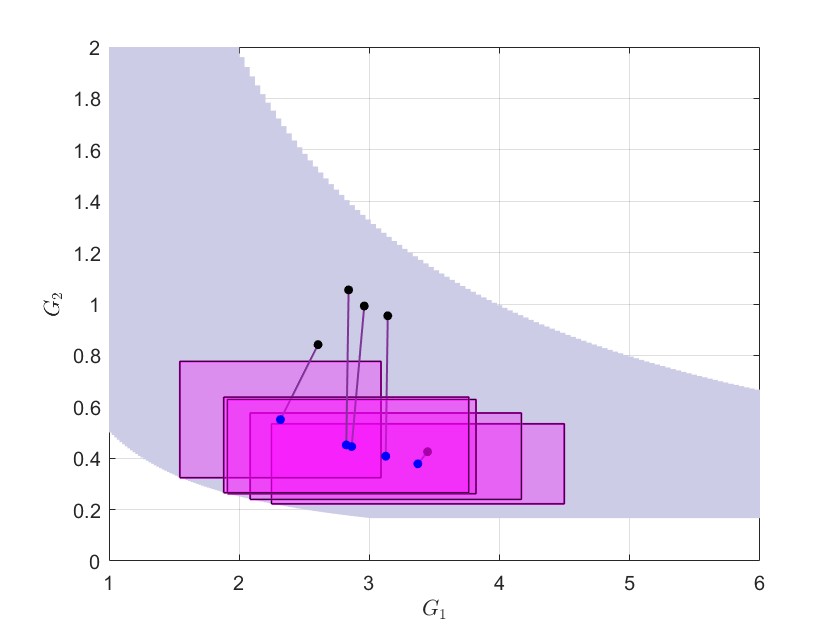}
			\caption{I-Deb}
			\label{fig:I-Deb}
		\end{subfigure}\hfill
		\begin{subfigure}[t]{0.33\textwidth}
			\includegraphics[width=\linewidth]{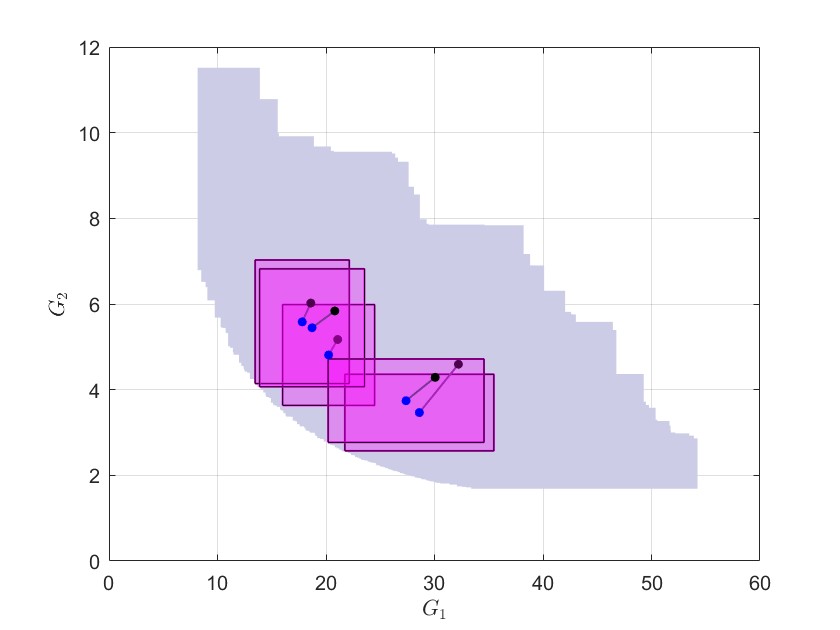}
			\caption{I-SD}
			\label{fig:I-SD}
		\end{subfigure}
		
		\caption{Output for biobjective test problems, initiated from five random starting points.}
		\label{figure:biobjective}
	\end{figure}
	
	\begin{figure}[htbp]
		\begin{subfigure}[t]{0.33\textwidth}
			\includegraphics[width=\linewidth]{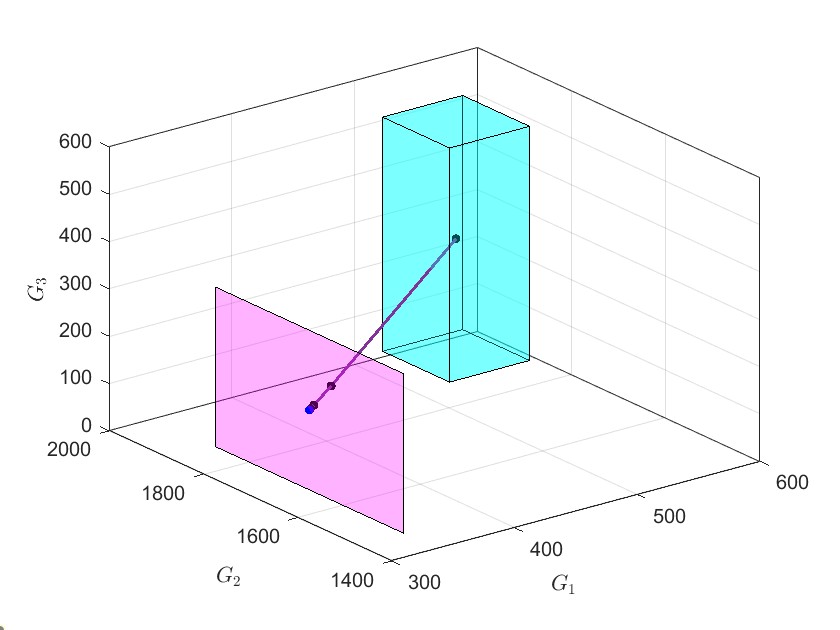}
			\caption{I-IKK1}
			\label{fig:I-IKK1}
		\end{subfigure}\hfill
		\begin{subfigure}[t]{0.33\textwidth}
			\includegraphics[width=\linewidth]{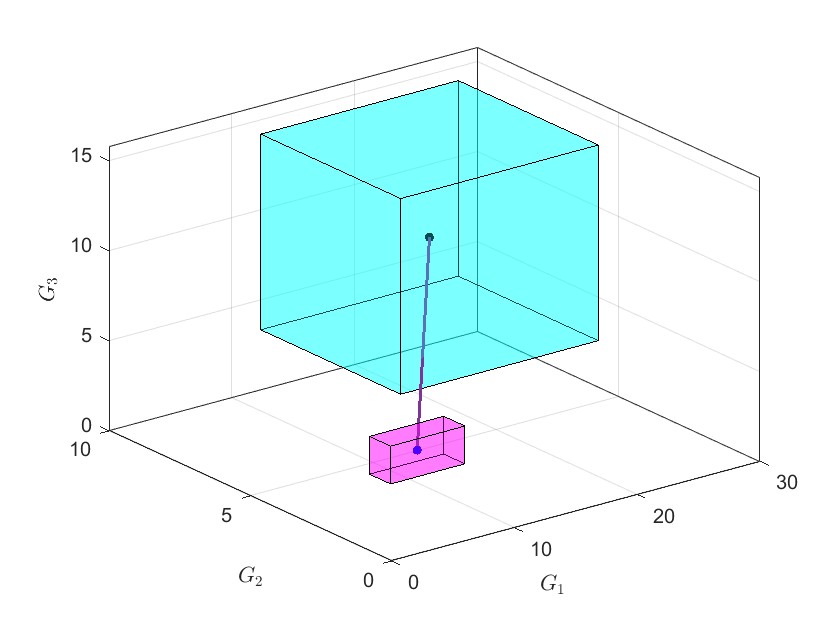}
			\caption{I-VFM1}
			\label{fig:I-VFM1}
		\end{subfigure}\hfill
		\begin{subfigure}[t]{0.33\textwidth}
			\includegraphics[width=\linewidth]{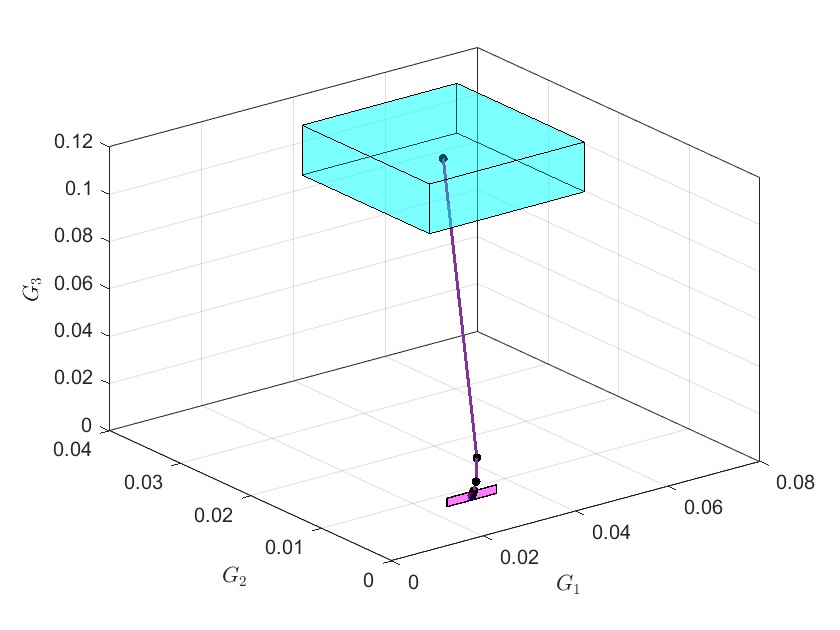}
			\caption{I-MHHM2}
			\label{fig:I-MHHM2}
		\end{subfigure}
		
		\begin{subfigure}[t]{0.33\textwidth}
			\includegraphics[width=\linewidth]{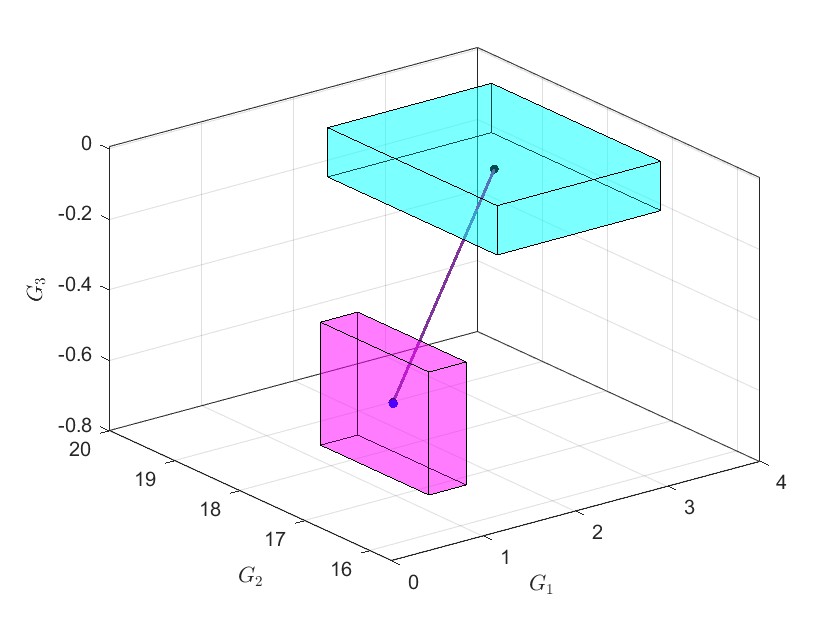}
			\caption{I-Viennet}
			\label{fig:I-Viennet}
		\end{subfigure}\hfill
		\begin{subfigure}[t]{0.33\textwidth}
			\includegraphics[width=\linewidth]{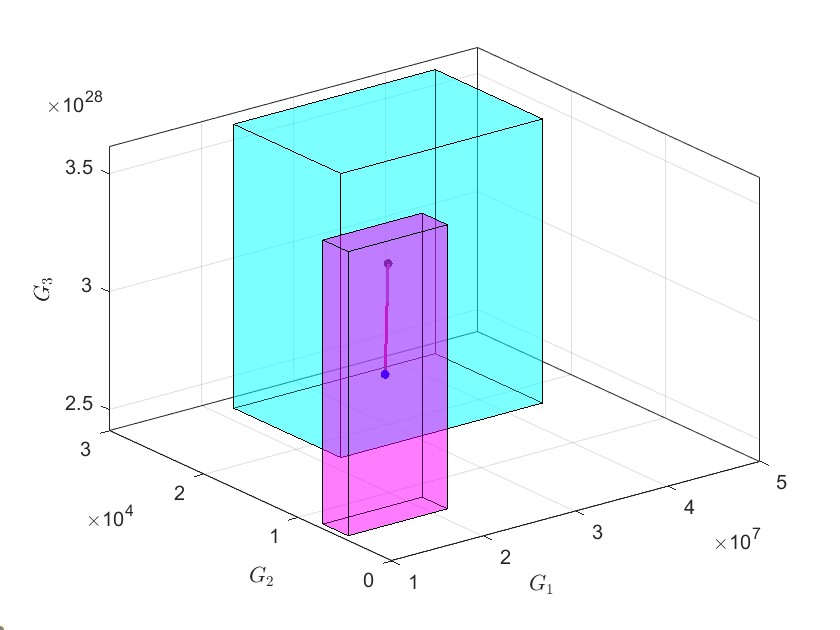}
			\caption{I-AP1}
			\label{fig:I-AP1}
		\end{subfigure}\hfill
		\begin{subfigure}[t]{0.33\textwidth}
			\includegraphics[width=\textwidth]{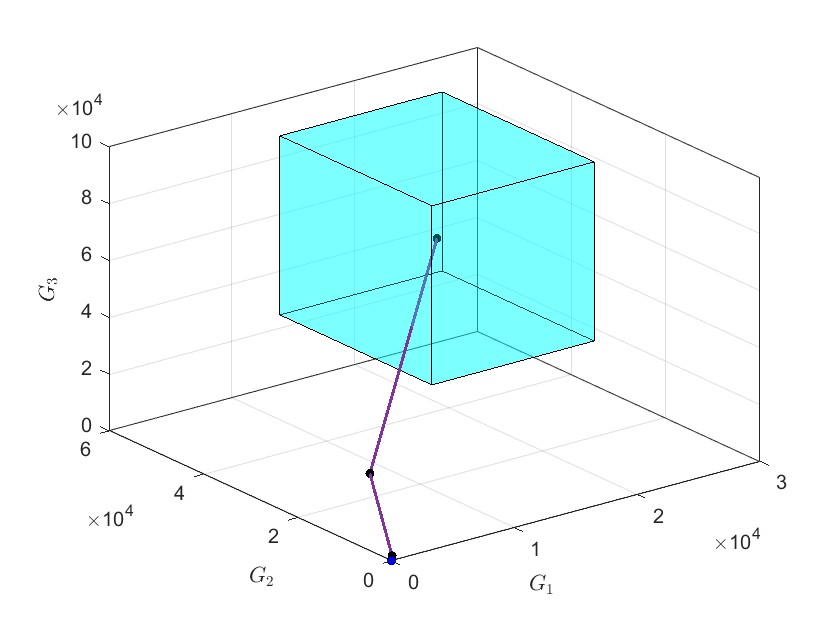}
			\caption{I-MOP7}
			\label{fig:I-MOP7}
		\end{subfigure}
		
		\begin{subfigure}[t]{0.33\textwidth}
			\includegraphics[width=\linewidth]{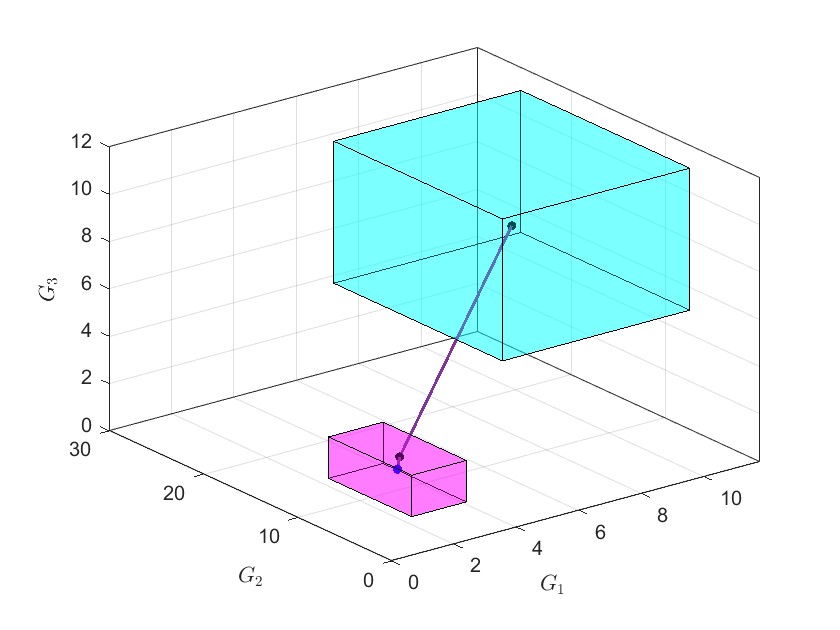}
			\caption{I-VFM2}
			\label{fig:I-VFM2}
		\end{subfigure}\hfill
		\begin{subfigure}[t]{0.33\textwidth}
			\includegraphics[width=\linewidth]{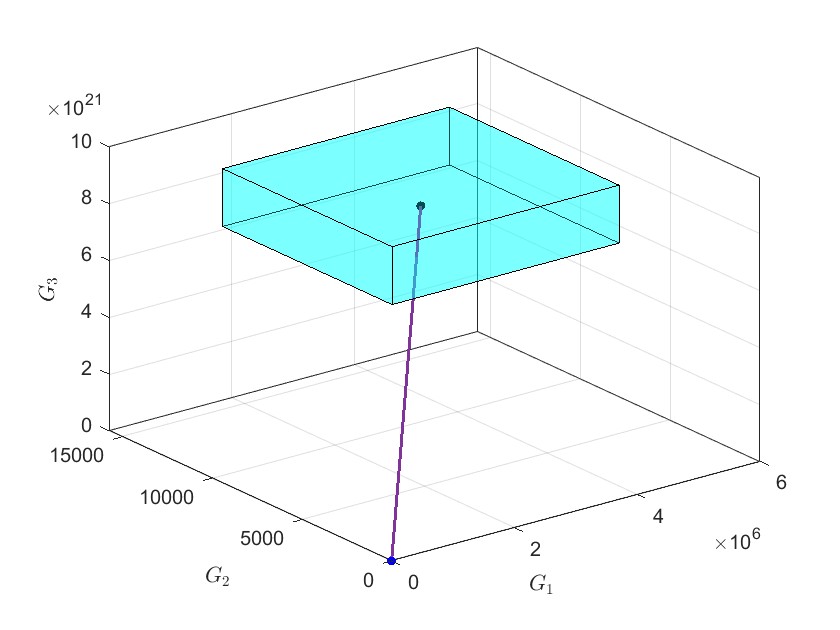}
			\caption{I-AP4}
			\label{fig:I-AP4}
		\end{subfigure}\hfill
		\begin{subfigure}[t]{0.33\textwidth}
			\includegraphics[width=\linewidth]{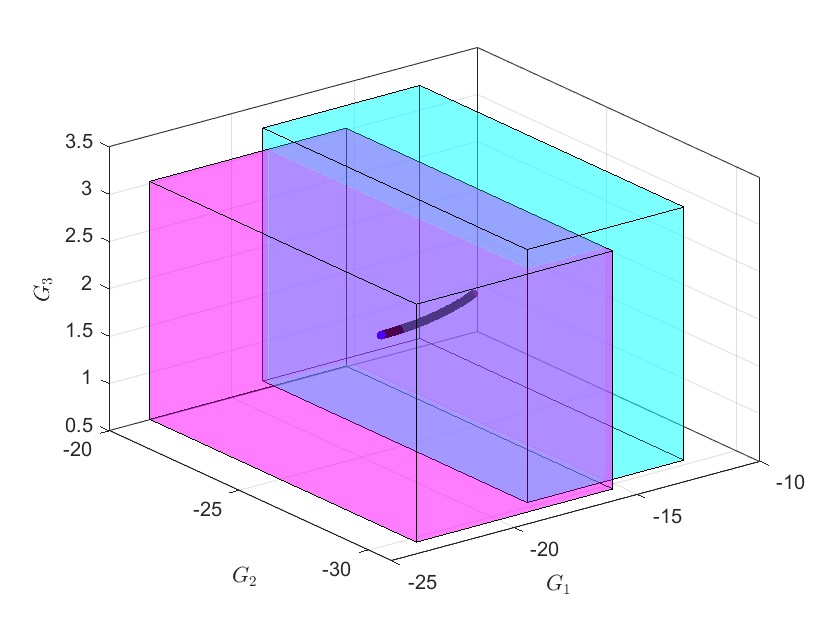}
			\caption{I-Comet}
			\label{fig:I-Comet}
		\end{subfigure}
		
		\caption{Initial and final objective values, $G(x^0)$ and $G(x^\star)$, for the triobjective test problems using a randomly initialized starting point.}
		\label{figure:triobjective}
	\end{figure}
\clearpage

\section{Conclusion and Future Directions}\label{Conclusion and Future Directions}
In this work, we have proposed three specialized variants of the NCGM, namely the PRP, HS, and LS schemes for solving unconstrained IVMOPs. First, we have provided the algorithm for the NCGM without any restriction on the algorithmic parameter (Algorithm \ref{Algorithm}). In Algorithm \ref{Algorithm}, we have used the standard Wolfe condition to compute an interval of step length. In addition, to derive the global convergence of our proposed Algorithm \ref{Algorithm}, we have assumed the sufficient decency condition. Under Assumptions \ref{assumption 1} and \ref{assumption 2}, we have established the global convergence of our proposed Algorithm \ref{Algorithm} when it satisfies Property \ref{property} with the sufficient decency condition \eqref{sufficient descent condition} and the standard Wolfe condition \eqref{Standard Wolfe condition in algorithm} (Theorem \ref{convergence theorem}). Further, we have established global convergence results for each of the proposed variants under appropriate assumptions, demonstrating the theoretical soundness of the approach (Theorem \ref{PRP,HS,LS convergence theorem}).

To assess the practical performance of the methods, we conducted a comprehensive numerical study using a collection of benchmark problems given in \cite{mondal2025steepest}. The proposed algorithms were compared with existing approaches, including the SD method \cite{mondal2025steepest} and other variants of the NCGM such as FR, CD, DY, and mDY \cite{mondal2026nonlinear}. The comparative analysis, carried out using the Dolan–Moré performance profile \cite{dolan2002benchmarking}, indicates that the PRP, HS, and LS variants exhibit improved efficiency in terms of convergence behavior and computational performance.

Despite these promising results, several avenues remain open for further investigation. One potential direction is the extension of the proposed methods to constrained IVMOPs, where feasibility conditions introduce additional complexity. Another important aspect is the development of adaptive or hybrid strategies that combine different conjugate gradient parameters to further enhance stability and convergence speed. Additionally, exploring second-order or quasi-Newton-type techniques within the IVMOP framework may provide further computational advantages. Under general strong convexity for IVFs, derivation of the Newton method is found difficult. Possibly, a dual reformulation \cite{peng2025optimality, peng2024optimality} of the IVF problems or $E$-convexity may help in this regard \cite{wen2025semi}.

Finally, applying the proposed algorithms to real-world problems involving uncertainty, such as engineering design, finance, and data-driven optimization, would be a valuable step toward demonstrating their practical applicability.

\bigskip

\noindent {\bf Acknowledgments} 
Tapas Mondal gratefully acknowledges the financial support provided by ANRF, India, under the Core Research Grant CRG/2022/001347 for carrying out this research. Zai-Yun Peng was partially supported by the National Natural Science Foundation of China (12271067) and the Chongqing Natural Science Foundation (CSTB2024NSCQ-MSX0973). Yong Zhao was supported by Natural Science Foundation 
of Chongqing (No. CSTB2025NSCQ-LZX0060) and the Science and Technology Research Program of Chongqing Education 
Commission  (No. KJQN202400760). \\

\noindent {\bf Data Availability}\\ 
No datasets were used in the preparation of this article. \\ 

\noindent {\bf Disclosure Statement} \\ 
The authors declare that there are no competing interests related to this work.


\begin{thebibliography}{plain}
	
\bibitem{assuncao2021conditional} Assun\c{c}\~{a}o,  P. B., Ferreira, O. P., Prudente, L. F.: Conditional gradient method for multiobjective optimization. Comput. Optim. Appl. {\bf 78}(3), 741--768 (2021).  
	
\bibitem{chauhan2021generalized} Chauhan, R. S., Ghosh, D., Ramík, J., Debnath, A. K.: Generalized Hukuhara-Clarke derivative of interval-valued functions and its properties. Soft Comput. {\bf 25}(23), 14629--14643 (2021).
 
 \bibitem{chen2025PRP} Chen, J., Bai, Y., Yu, G., Ou, X., Qin, X.: A PRP type conjugate gradient method without truncation for nonconvex vector optimization. J. Optim. Theory Appl. {\bf 204}, 13 (2025).
 	

\bibitem{debnath2022generalized} Debnath, A. K., Ghosh, D., Mesiar, R., Chauhan, R. S.: Generalized-Hukuhara subgradient and its application in optimization problem with interval-valued functions. S$\bar{\textnormal{a}}$dhan$\bar{\textnormal{a}}$ {\bf 47}(2), 1--16 (2022). 
	
\bibitem{dolan2002benchmarking} Dolan, E. D., Mor{\'e}, J. J.: Benchmarking optimization software with performance profiles. Math. Program., Ser. A {\bf 91}, 201--213 (2002).

\bibitem{drummond2004projected} Drummond, L. M. G., Iusem, A. N.: A projected gradient method for vector optimization problems.
Comput. Optim. Appl. {\bf 28}(1), 5--29 (2004). 


\bibitem{ehrgott2005multicriteria}  Ehrgott, M.: Multicriteria Optimization, Second Edition, Springer, Berlin Heidelberg New York, (2005). 


\bibitem{Elboulqe2024explicit} Elboulqe, Y., Maghri, M. E.: An explicit three-term Polak–Ribière–Polyak conjugate gradient method for bicriteria optimization. Oper. Res. Lett. {\bf 57}, 107195 (2024).
	
\bibitem{fliege2009newton} Fleige, J., Drummond, L. M. G., Svaiter, B. F.: Newton’s method for multiobjective optimization. SIAM J. Optim. {\bf 20}(2), 602--626 (2009). 

\bibitem{fliege2000steepest}  Fliege, J., Svaiter, B. F.: Steepest descent methods for multicriteria optimization. Math. Methods Oper.
Res. {\bf 51}(3), 479--494 (2000). 

\bibitem{ghosh2022generalized} Ghosh, D., Debnath, A. K., Chauhan, R. S., Castillo, O.: Generalized-Hukuhara-gradient efficient-direction method to solve optimization problems with interval-valued functions and its application in least-square problems. Int. J. Fuzzy Syst. {\bf 24}(3), 1275--1300 (2022). 

\bibitem{ghosh2017newton} Ghosh, D.: Newton method to obtain efficient solutions of the optimization problems with interval-valued objective functions. J. Appl. Math. Comput. {\bf 53}(1-2), 709--731 (2017).  

\bibitem{ghosh2017quasi} Ghosh, D.: A quasi-Newton method with rank-two update to solve interval optimization problems. Int. J. Comput. Appl. Math. {\bf 3}, 1719--1738 (2017).
	

\bibitem{ghosh2014directed} Ghosh, D., Chakraborty, D.: A direction based classical method to obtain complete Pareto set of multi-criteria optimization problems. Opsearch {\bf 52}(2), 340--366 (2015). 

\bibitem{Gilbert1992global} Gilbert, J. C., Nocedal, J.: Global convergence properties of conjugate gradient methods for optimization. SIAM J. Optim. {\bf 2}, 21--42 (1992).

\bibitem{gonclaves2022study} Gon\c{c}alves, M. L. N., Lima, F. S., Prudente, L. F.: A study of Liu-Storey conjugate gradient methods for vector optimization. Appl. Math. Comput. {\bf 425}, 127099 (2022). 



\bibitem{Hu2025modified} Hu, Q., Zhang, Y., Li, R., Zhu, Z.: A modified Polak-Ribière-Polyak type conjugate gradient method for vector optimization. Optim. Methods Softw. {\bf 40}(4), 725--754 (2025).

\bibitem{lapucci2023limited}Lapucci, M., Mansueto, P.:  A limited memory quasi-newton approach for multi-objective
optimization. Comput. Optim. Appl. {\bf 85}, 33--73 (2023).

\bibitem{miettinen1999nonlinear} Miettinen, K.: Nonlinear Multiobjective Optimization. Kluwer Academic Publishers, New York (1999).   

\bibitem{mohammadi2024trust} Mohammadi, A., Cust\'{o}dio, A. L.: A trust-region approach for computing Pareto fronts in multiobjective optimization. Comput. Optim. Appl. {\bf 87}(1), 149--179 (2024).    
	
\bibitem{mondal2025steepest} Mondal, T., Ghosh, D.: Steepest descent method for multiobjective optimization problems of interval-valued maps. Numer. Algorithms (2025). {\color{blue}https://doi.org/10.1007/s11075-025-02205-7}

\bibitem{mondal2026nonlinear} Mondal, T., Ghosh, D., Liu, J., Li, J.: Nonlinear conjugate gradient method for multiobjective optimization problems of interval-valued maps. Communicated. \\ 
{\color{blue}https://doi.org/10.48550/arXiv.2603.05814}


\bibitem{moore1966interval} Moore, R. E.: Interval Analysis. Prentice-Hall, Englewood Cliffs (1966). 






\bibitem{Peng2025novel} Peng, J., W., Zhong, D. H., Singh, A.: A novel modified Liu-Storey nonlinear conjugate gradient method for solving vector optimization problems. Optimization (2025). \\{\color{blue}https://doi.org/10.1080/02331934.2025.2516452}


\bibitem{peng2025optimality}
Peng, Z. Y., Peng, J. Y., Ghosh, D., Zhao, Y., Li, D.: Optimality conditions and duality results for generalized-Hukuhara subdifferentiable preinvex interval-valued vector optimization problems, Fuzzy Set Syst., {\bf 515}, 109416 (2025) 


\bibitem{peng2024optimality}
Peng, Z. Y., Deng, C. Y., Zhao, Y., Peng, J. Y.: 
Optimality conditions and duality for E-differentiable fractional multiobjective interval valued optimization problems with E-invexity, 
Set Valued Anal. Optim, 
{\bf 6}(3), 295--307 (2024). 





\bibitem{perez2018nonlinear} P{\'e}rez, L. R. L., Prudente, L. F.: Nonlinear conjugate gradient methods for vector optimization. SIAM J. Optim. {\bf 28}(3), 2690--2720 (2018).

\bibitem{polvaj2014quasi} Povalej, \v Z.: Quasi-Newton’s method for multiobjective optimization. J. Comput. Appl. Math. {\bf 255}, 765--777 (2014). 

\bibitem{stefanini2008generalization} Stefanini, L.: A generalization of Hukuhara difference. In: Dubois, D., Lubiano, M. A., Prade, H., Gil, M. Á., Grzegorzewski, P., Hryniewicz, O. (eds) Soft Methods for Handling Variability and Imprecision. Advances in Soft Computing, {\bf 48}, pp. 203--210, Springer, Berlin, Heidelberg, (2008).  






\bibitem{wen2025semi}
Wen, M., Peng Z. Y., Tan X. Y., Deng C. Y.: $E$-semi-preinvex interval-valued functions and interval-valued programming. Journal of Chongqing Normal University (Natural Science Edition), {\bf 42}(2),  117--126 (2025). 


\bibitem{wu2007karush} Wu, H. C.: The Karush–Kuhn–Tucker optimality conditions in an optimization problem with interval-valued objective function. Eur. J. Oper. Res. {\bf 176}(1), 46--59 (2007). 




\end{thebibliography}
\end{document}